\documentclass[11pt,reqno]{amsart}   
\def\MR#1{}
\usepackage[margin=37mm]{geometry}
\usepackage{amsfonts,amsmath,amssymb,enumerate}    
\usepackage{amsthm, bm}    
\usepackage{mathrsfs}
\usepackage{setspace}

\usepackage{comment}

\usepackage{lmodern} % scalable version of Computer Modern
\usepackage[spacing=true,kerning=true]{microtype} % reduce overfull lines

\usepackage{pgfplots}
\usepackage{tikz-3dplot}
\usetikzlibrary{arrows,matrix,shadows.blur}
\usetikzlibrary{decorations.markings,shapes.geometric,shapes.misc,cd} 
\usetikzlibrary{3d}
\usetikzlibrary {graphs.standard}
\pgfplotsset{compat=1.17}
\usepgfplotslibrary{fillbetween}

\usepackage{mathtools}    
\usepackage{hyperref}
\usepackage{todonotes}

\usepackage[capitalise]{cleveref}
\crefname{subsection}{\S\!}{\S\!}
\Crefname{subsection}{Section}{Sections}
\crefname{subappendix}{\S\!}{\S\!}
\Crefname{subappendix}{Section}{Sections}
\usepackage{upgreek}
\usepackage{simpler-wick}

\usepackage[hyperref,doi,url=false,
backref,
giveninits=true,
backend=biber,
style=alphabetic,
maxbibnames=99]{biblatex}
\bibliography{bibliography.bib,MyLibrary.bib}

\AtEveryBibitem{% Clean up the bibtex rather than editing it
 \clearlist{address}
 \clearfield{date}
 \clearfield{isbn}
 \clearfield{issn}
 \clearlist{location}
 \clearfield{month}
 \clearfield{series}
 
 \ifentrytype{book}{}{% Remove publisher and editor except for books
  \clearlist{publisher}
  \clearname{editor}
 }
 \ifentrytype{misc}{}{% Remove arxiv number except for preprints
  \clearfield{eprint}
 }
}

\crefname{equation}{}{}
\Crefname{equation}{}{}

\DeclareFontEncoding{LS1}{}{}
\DeclareFontSubstitution{LS1}{stix}{m}{n}
\DeclareSymbolFont{symbols2}{LS1}{stixfrak}{m}{n}
\DeclareMathSymbol{\typecolon}{\mathbin}{symbols2}{"25}
\theoremstyle{plain}% default
\newtheorem{thm}{Theorem}
\newtheorem{lem}[thm]{Lemma}
\newtheorem{prop}[thm]{Proposition}

\newtheorem*{prop*}{Proposition}
\newtheorem*{thm*}{Theorem}
\theoremstyle{definition}

\newtheorem{defprop}[thm]{Definition-Proposition}

\newtheorem{examplex}[thm]{Example}
\newenvironment{exmp}
  {\pushQED{\qed}\examplex}
  {\popQED\endexamplex}
\theoremstyle{remark}
\newtheorem{remarkx}[thm]{Remark}
\newenvironment{rem}
  {\pushQED{\qed}\remarkx}
  {\popQED\endremarkx}

\newtheorem*{rem*}{Remark}
\newtheorem*{ack}{Acknowledgements}

\newcommand{\be}{\begin{equation}}    
\newcommand{\ee}{\end{equation}}    
\newcommand{\bmx}{\begin{pmatrix}}    
\newcommand{\emx}{\end{pmatrix}}    

\newcommand{\bmb}{\begin{matrix}}    
\newcommand{\bme}{\end{matrix}}

\newcommand{\bbb}{\!\begin{bmatrix}}    
\newcommand{\bbe}{\end{bmatrix}}    

\newcommand{\ttb}{\!\left[\!\!\left[\begin{matrix}}
\newcommand{\tte}{\end{matrix}\right]\!\!\right]}

\newcommand{\llb}{\hspace{-.15em}\left(\!\!\!\left(\begin{matrix}}
\newcommand{\lle}{\end{matrix}\right)\!\!\!\right)}

\newcommand{\ol}{\overline}    
    
\newcommand{\del}{\partial}    
\newcommand{\g}{{\mathfrak g}}
\renewcommand{\a}{{\mathfrak a}}
\renewcommand{\b}{{\mathfrak b}}

\newcommand{\B}{{\mathsf B}}

\newcommand{\mf}{\mathfrak}
\newcommand{\mc}{\mathcal}

\newcommand{\half}{\frac{1}{2}}

\newcommand{\nn}{\nonumber}

\newcommand{\8}{{\infty}}

\newcommand{\ZZ}{{\mathbb Z}}
\newcommand{\N}{{\mathbb N}}
    
\newcommand{\CC}{{\mathbb C}}

\renewcommand{\AA}{{\mathbb A}}

\newcommand{\QQ}{\mathbb{Q}}    
    
\newcommand{\ket}[1]{{\,\left|#1\right>}\,}

\newcommand{\id}{{\textup{id}}}    
    
\newcommand{\wh}{\widehat}

\newcommand{\Qc}[2]{\pi_i \circ Q \circ \pi_j}

\newcommand{\A}{\mathcal A}

\newcommand{\goi}[2]{=}    
\newcommand{\Hom}{\mathrm{Hom}}

\newcommand{\on}{.}

\newcommand{\btp}{\begin{tikzpicture}[baseline=0pt,scale=0.9,line width=0.25pt]}    
\newcommand{\etp}{\end{tikzpicture}}

\DeclareMathOperator{\ev}{ev}

\DeclareMathOperator{\Spec}{Spec}

\renewcommand{\O}{\mc O}

\newcommand{\VV}{{\mathbb V}}

\newcommand{\MM}{{\mathbb M}}

\newcommand{\vac}{\!\ket{0}\!}

\DeclareMathOperator{\Ind}{Ind}

\DeclareMathOperator{\End}{End}

\newcommand{\ox}{\mathbin\otimes}

\newcommand{\bul}{\bullet}

\newcommand{\atp}[2]{\underset{\substack{$ $ \\ \tikz{%\filldraw[black] (0,.13) circle (1pt); 
\draw[-stealth] (0,0) -- (0,.13);} \\[-.2mm] {#1}}}{\smash{#2}}}

\newcommand{\into}{\hookrightarrow}

\newcommand{\onto}{\twoheadrightarrow}

\def\N{N}

\def\B{\mc B}

\newcommand{\op}{\mathrm{op}}

\newcommand{\x}{{\bm x}}

\newcommand{\coinv}{\mathsf F}

\DeclareMathOperator{\Sym}{Sym}

\newcommand{\C}{\mathcal C}

\makeatletter
\newcommand{\extp}{\@ifnextchar^\@extp{\@extp^{\,}}}
\def\@extp^#1{\mathop{\bigwedge\nolimits^{\!#1}}}
\makeatother
\makeatletter
\newcommand{\hextp}{\@ifnextchar^\@hextp{\@hextp^{\,}}}

\def\@hextp^#1{\mathop{\wh{\bigwedge\nolimits^{\!#1}}}}
\makeatother

\newcommand{\catname}[1]{\mathbf{#1}}

\newcommand{\Operad}[1]{\mathcal #1}

\newcommand{\LieOperad}{\Operad{Lie}}
\newcommand{\ComOperad}{\Operad{Com}}
\newcommand{\AssOperad}{\Operad{As}}

\newcommand{\CAlg}{\catname{Alg}^{\ComOperad}}
\newcommand{\LieAlg}{\catname{Alg}^{\LieOperad}}
\newcommand{\AssAlg}{\catname{Alg}^{\AssOperad}}

\newcommand{\dgCAlg}{\CAlg(\dgVect_\CC)}

\newcommand{\dgVect}{\catname{dgVect}}
\newcommand{\grVect}{\catname{grVect}}
\newcommand{\Vect}{\catname{Vect}}

\newcommand{\RavConf}{\mathrm{RavConf}}
\newcommand{\Conf}{\mathrm{Conf}}

\newcommand{\Mod}{\catname{Mod}}

\newcommand{\dd}{\mathrm{d}}

\newcommand{\Disc}{\mathrm{Disc}}
\newcommand{\Discp}{\mathrm{Disc}^\times}

\newcommand{\Unshf}[2]{\mathrm{Unshf}_{#1}^{#2}}

\DeclareMathOperator{\holim}{holim}

\newcommand{\Cech}{\check C}

\DeclareMathOperator{\Th}{Th}

\newcommand{\iglob}{I_{\mathrm{Global}}}
\newcommand{\iloc}{I_{\mathrm{Ravioli}}}

\newcommand{\gm}{\g_-}

\newcommand{\dfn}[1]{{\color{blue}\emph{#1}}}

\newcommand{\lox}{\otimes^{\mathbb L}}

\newcommand{\Sch}{\catname{Sch}}

\newcommand{\ijN}{_{\substack{1\leq i,j\leq N \\ i\neq j}}}
\newcommand{\ijNp}{_{1\leq i<j\leq N}}

\newcommand{\iotafar}{\iota_{\text{far}}}
\newcommand{\iotanear}{\iota_{\text{near}}}

\newcommand{\iotacob}{\iota_{\!\!\!\substack{\textup{base}\\{\textup{change}}}}}

\newcommand{\sql}{\{\!\{}
\newcommand{\sqr}{\}\!\}}

\newcommand{\Partial}{\mathrm{Partial}}
\newcommand{\Rav}{\mathrm{Rav}}
\renewcommand{\a}{\mf a}

\DeclareMathOperator{\colim}{colim}

\author{Luigi Alfonsi, Hyungrok Kim and Charles Young}
\address{%\emph{Email address:} \texttt{c.a.s.young@gmail.com}}
Department of Physics, Astronomy and Mathematics, University of Hertfordshire, College Lane, Hatfield AL10 9AB, UK.}  \email{l.alfonsi@herts.ac.uk, h.kim2@herts.ac.uk, c.young8@herts.ac.uk} 
\hypersetup {
    pdftitle = {Raviolo vertex algebras, cochains and conformal blocks},
    pdfauthor = {Luigi Alfonsi, Hyungrok Kim, and Charles Alastair Stephen Young},
    pdfsubject = {hep-th, math-ph, 17B67, 17B69, 81R10, 81R12},
    pdflang = {en-GB-oed},
    pdfkeywords = {raviolo vertex algebra, topological-holomorphic field theory, homotopy algebra, C∞-algebra, conformal blocks, coinvariants}
}
\title[Raviolo vertex algebras, cochains and conformal blocks]{Raviolo vertex algebras, cochains\\ and conformal blocks}

\begin{document}

\begin{abstract}
Raviolo vertex algebras were introduced recently by Garner and Williams in \cite{GarneWilliRavioloVertexAlgebras2023}. Working at the level of cochain complexes, in the present paper we construct spaces of conformal blocks, or more precisely their duals, coinvariants, in the raviolo setting. We prove that the raviolo state-field map correctly captures the limiting behaviour of coinvariants as marked points collide.
\end{abstract}

%17B67  	Kac-Moody (super)algebras; extended affine Lie algebras; toroidal Lie algebras
%17B69  	Vertex operators; vertex operator algebras and related structures
%81R10  	Infinite-dimensional groups and algebras motivated by physics, including Virasoro, Kac-Moody, $W$-algebras and other current algebras and their representations
%81R12  	Groups and algebras in quantum theory and relations with integrable systems

\maketitle
\setcounter{tocdepth}{2}
\tableofcontents

\section{Introduction}
Vertex algebras capture %in a mathematically precise form 
the physicists' notion of operator product expansions and the state-field correspondence in chiral conformal field theory in one complex dimension. Since their introduction \cite{BorchVertexAlgebrasKacMoody1986} they have become a powerful and ubiquitous tool in mathematical physics and representation theory. Textbook references include \cite{KacVertexAlgBook,LepowLiIntroductionVertexOperator2004,FrenkelBenZvi}. 

The definition of vertex algebras appears to be closely tied to the special properties of complex dimension one, and specifically of the formal disc $D$ and the punctured formal disc $D^\times = D \setminus\{0\}$,
\[ D = \Spec\CC[[z]],\qquad D^\times = \Spec \CC((z)). \]
Roughly speaking, the punctured formal disc $D^\times$ describes the possible collision geometries of two marked points in the complex plane (one fixed at the origin, the other movable). See \cref{fig: discs and rav}.

What happens in higher dimensions? 
It has long been expected by experts that the language of factorization algebras (as developed in the algebro-geometric setting in \cite{FrancGaitsChiralKoszulDuality2012} following \cite{BDChiralAlgebras} -- cf. also \cite{DeSoKacLieConformalAlgebra2009,BakalDeSoHeluaKacOperadicApproachVertex2019a,BakalDeSoKacComputationCohomologyVertex2021} -- and in the smooth setting in \cite{CG1,CG2}) in principle allows vertex algebras to generalize to higher dimensions. 
In this direction, see especially \cite{zotero-2777}, \cite[\S4]{WilliamsThesis} and \cite{GWHigherKM}, and also \cite{SWW}, all broadly in the smooth setting and using (pre)factorization algebras constructed using the Dolbeault resolution of the holomorphic structure sheaf; and in the complex-algebraic setting see \cite{FaontHenniKapraHigherKacMoody2019}, \cite{HenniKapraGelFandFuchsCohomology2023}, \cite{KapraInfinitedimensionalDgLie2021}. 
Writing down explicit axioms in closed form for higher-dimensional vertex algebras remains a challenge, however.

Recently though, Garner and Williams (\cite{GarneWilliRavioloVertexAlgebras2023} and see also \cite{GarneRaghaWilliHiggsCoulombBranches2023}) have considered a particularly tractable instance of this general problem, namely the case of theories with one real topological dimension and one complex holomorphic dimension, i.e.~theories on spacetimes modelled on \(\mathbb R\times\mathbb C\), such as twists of three-dimensional supersymmetric Yang-Mills theory \cite{elliott2020taxonomy}. Such a spacetime structure can be neatly captured by a transversely holomorphic foliation \cite{DK79,Raw79,Asu10}: that is, a foliation of the spacetime three-manifold by curves such that the leaf space has the structure of a Riemann surface.

In that topological-holomorphic setting, marked points are allowed to collide in the complex plane, but only if, when they do, they are separated in the topological direction. The upshot is that pairwise collisions are no longer described by the formal punctured disc $D^\times$, but rather by the \dfn{formal raviolo} \[\Rav := D \sqcup_{D^\times} D,\] the scheme obtained by gluing together two copies of the formal disc along their common copy of the formal punctured disc. Again, see \cref{fig: discs and rav}, and \cref{the-formal-raviolo} below.

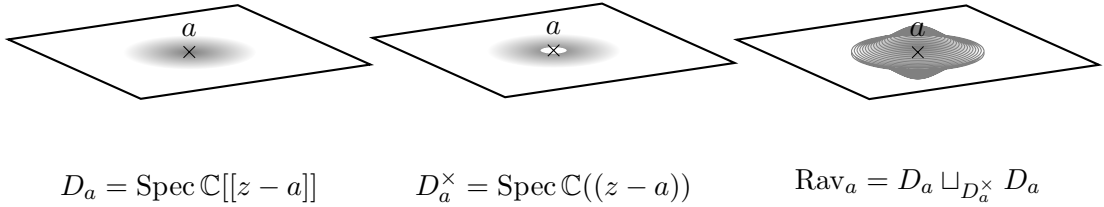
\begin{figure}
  \[
  \begin{tikzpicture}
    \begin{axis}[scale=0.75,axis equal image, hide axis,view={60}{15},clip mode = individual,zmin=-1,zmax=1,width=8cm,height=6cm,xmin=-2,xmax=2,ymin=-2,ymax=2,local bounding box=Da]
      \addplot3 [name path = outer,data cs=polarrad,domain=0:2*pi,draw=none,smooth] (\x,1,0);
      \addplot3 [name path = inner,data cs=polarrad,domain=0:2*pi,draw=none,smooth] (\x,0,0);
      \addplot3 [shading=radial, fill=white, fill opacity=0.5] fill between [of=outer and inner];
  %  \addplot3 [shading=radial, fill=blue, fill opacity=0.5] fill between [of=inner and outer];
      \draw[black,thick] (axis cs:-2,-2,0) -- (axis cs:2,-2,0) -- (axis cs:2,2,0) -- (axis cs:-2,2,0) -- cycle;
    %  \node[fill=white,fill] at (axis cs:2,2,0) {$\AA_\CC^1_\CC$};
      \node[cross out,draw=black, fill=white, inner sep =2,label=above:$a$] at (axis cs:0,0,0) {} ;
    \end{axis}
    \path (Da.south) node[below]{$D_a = \Spec\CC[[z-a]]$};
  \end{tikzpicture}
  \begin{tikzpicture}
    \begin{axis}[scale=0.75,axis equal image, hide axis,view={60}{15},clip mode = individual,zmin=-1,zmax=1,width=8cm,height=6cm,xmin=-2,xmax=2,ymin=-2,ymax=2,local bounding box=Dax]
      \addplot3 [name path = outer,data cs=polarrad,domain=-pi:pi,draw=none,smooth] (\x,1,0);    
      \addplot3 [name path = inner,data cs=polarrad,domain=-pi:pi,draw=none,smooth] (\x,.2,0);
      \addplot3 [shading=radial, fill=white, fill opacity=0.5] fill between [of=outer and inner];
  %  \addplot3 [shading=radial, fill=blue, fill opacity=0.5] fill between [of=inner and outer];
      \draw[black,thick] (axis cs:-2,-2,0) -- (axis cs:2,-2,0) -- (axis cs:2,2,0) -- (axis cs:-2,2,0) -- cycle;
    %  \node[fill=white,fill] at (axis cs:2,2,0) {$\AA_\CC^1_\CC$};
      \node[cross out,draw=black, fill=white, inner sep =2,label=above:$a$] at (axis cs:0,0,0) {} ;
    \end{axis}
    \path (Dax.south) node[below]{$D^\times_a = \Spec\CC((z-a))$};
  \end{tikzpicture}
  \begin{tikzpicture}
    \begin{axis}[scale=0.75,axis equal image, hide axis,view={60}{15},clip mode = individual,
      zmin=-1,zmax=1,width=8cm,height=6cm,xmin=-2,xmax=2,ymin=-2,ymax=2,local bounding box=Rav]
      \addplot3 [name path = lower,data cs=polarrad,fill=white,gray,domain=-pi:pi,y domain=0:1,shading=radial,smooth] 
      (\x,\y,{-0.4*exp(-3*\y^2)});
      \addplot3 [name path = upper,data cs=polarrad,fill=blue,gray,domain=-pi:pi,y domain=0:1,shading=radial,smooth] 
      (\x,\y,{0.4*exp(-3*\y^2)});
      \draw[black,thick] (axis cs:-2,-2,0) -- (axis cs:2,-2,0) -- (axis cs:2,2,0) -- (axis cs:-2,2,0) -- cycle;
    %  \node[fill=white,fill] at (axis cs:2,2,0) {\fontsize{12}{14}\selectfont $\AA_\CC^1_\CC$};
      \node[cross out,draw=black, fill=white, inner sep =2,label=above:{\fontsize{12}{14}\selectfont $a$}] at (axis cs:0,0,0) {} ;
    \end{axis}
    \path (Rav.south) node[below]{$\Rav_a = D_a \sqcup_{D^\times_a} D_a$};
  \end{tikzpicture}
  \]
  \caption{\label{fig: discs and rav} Sketch of copies of the formal disc $D$, the formal punctured disc $D^\times$, and the formal raviolo $\Rav$ associated to a point $a$ in the complex plane.}
  \end{figure}

What's so nice about this setting is that, on the one hand, it is a sufficiently mild generalization that it is still possible to write down explicit axioms for the resulting \dfn{raviolo vertex algebras} in a form closely parallel to the usual case -- see \cite{GarneWilliRavioloVertexAlgebras2023} -- while on the other hand it is sufficiently different that it exhibits many of the features expected in the higher setting. Indeed, the formal raviolo $\Rav$ is no longer an affine scheme, unlike the punctured disc $D^\times$. Its structure sheaf has higher cohomology, and one consequence is that in the raviolo vertex algebras of \cite{GarneWilliRavioloVertexAlgebras2023} the usual lowering operators/negative modes disappear from degree zero, and reappear in cohomogical degree one.

\bigskip

Now let us describe the contents of the present paper. We do essentially two things: first, we work at the level of cochain complexes rather than their cohomologies, and second, we introduce notions of configuration space and rational conformal blocks in the raviolo setting. 

In \cite{GarneWilliRavioloVertexAlgebras2023}, the sheaf cohomology $H^\bul(\Rav,\O)$ is regarded as a graded commutative algebra. It plays the role, there, that commutative algebra of functions on the punctured formal disc, $\Gamma(D^\times,\O) = H^0(D^\times,\O) \cong \CC((z))$, plays in the case of standard vertex algebras -- i.e., roughly speaking, it is what organizes the positive and negative modes in the state-field correspondence.
However, the cohomology $H^\bul(\Rav,\O)$ comes endowed with additional higher structure which is  lost in this picture (as the authors of \cite{GarneWilliRavioloVertexAlgebras2023} remark). One way to keep track of that higher structure is to work, instead, with the \dfn{derived global sections} $R\Gamma^\bul(\Rav,\O)$ of the structure sheaf. This latter comes with the structure of a differential graded (dg) commutative algebra, which, via homotopy transfer, encodes all the higher products on its cohomology,
\[ H^\bul(\Rav,\O) \equiv H^\bul(R\Gamma(\Rav,\O)).\]
Thus, for us it will be (a certain model of) $R\Gamma^\bul(\Rav,\O)$ which plays the role of $\Gamma(D^\times,\O)=\CC((z))$ in the usual case. For simplicity we focus exclusively on  the raviolo analogues of the Kac-Moody vertex algebras at level zero. 
We write down in \cref{sec: the raviolo vacuum module in cochain complexes} the definition of the raviolo vacuum module and its state-field map at the level of cochain complexes (i.e. dg vector spaces) rather than graded vector spaces. One immediate consequence is that for us there are lowering operators \emph{both} in degree one (representing cohomology classes) \emph{and} in degree zero; see the discussion in \cref{sec: raviolo state-field map}.

Then our second and main goal is to introduce rational conformal blocks (or more precisely, their duals, rational coinvariants) in the raviolo setting. 
To do that, after reviewing the standard definition of rational coinvariants in \cref{sec: usual rational conformal blocks}, we define in \cref{sec: configuration spaces of ravioli} a notion of \dfn{ravioli configuration space}, $\RavConf_N$, which plays the role of the usual configuration space 
\[ \Conf_N = \AA_\CC^N \setminus \text{diagonals} \] 
for standard vertex algebras. Mirroring the passage from the punctured disc $D^\times$ to the formal raviolo $\Rav$, the ravioli configuration space will be defined by gluing together $N!$ copies of $\AA_\CC^N$ along the complements of their diagonals. The resulting (non-separated) scheme looks locally like $\AA_\CC^N$ everywhere except on diagonals, just like $\Conf_N$. But the diagonals themselves, rather than being removed, instead appear with multiplicity $>1$ (as one expects since, whenever marked points coincide, one has to keep track of their ordering in the topological/leaf direction). 
We introduce a model of the derived space of sections 
\[ \A_N \simeq R\Gamma(\RavConf_N,\O) \]
of the structure sheaf on this configuration space. It is a dg commutative algebra, and it plays the role of the commutative algebra of functions $\B_N := \Gamma(\Conf_N,\O) =\CC[z_i,(z_i-z_j)^{-1}]_{\substack{i,j=1; i\neq j}}^N$ on configuration space in the usual case. 

Our model $\A_N$ is so chosen that it is possible to write down (see \cref{sec: rational ravioli coinvariants}) explicit raviolo analogues of all the constructions reviewed in \cref{sec: usual rational conformal blocks} for the usual case. 
%In particular, we obtain an example of what was called a \dfn{homotopy Manin triple} in \cite{AlfonYoungHigherCurrentAlgebras2023a}; see \cref{sec: lie-algebra-splitting}. 
We arrive at the definition of the space -- more precisely, the dg $\A_N$-module -- of \dfn{ravioli coinvariants}
\[ \coinv(\g;\mc A_N; M_1,\dots,M_N).  \] 
(See \cref{sec: induced modules and coinvariants}.)

The main result of the paper is then \cref{thm: raviolo Y map} in \cref{sec: main result}, which shows that the state-field map for the raviolo vacuum module, defined in \cref{sec: the raviolo vacuum module in cochain complexes}, correctly captures the limiting behaviour of ravioli coinvariants as two marked points, each with copies of the vacuum module attached, are brought close together. 

As we discuss at greater length in \cref{sec: raviolo state-field map} below, the limiting behaviour of conformal blocks as two or three marked points collide is arguably what motivates the usual vertex algebra axioms (notably, Borcherds identity), and at any rate those axioms can certainly be reconstructed by considering such limits. \cref{thm: raviolo Y map} establishes an analogous setup in the raviolo case.

\bigskip

The proof of the main theorem, \cref{thm: raviolo Y map}, is given in \cref{sec: proof of thm raviolo Y map}. Finally, in an appendix, %we briefly discuss higher products on the cohomology of the raviolo structure sheaf, and 
we recall some background material on semisimplicial sets and the Thom-Sullivan functor.

\bigskip
\begin{ack}
The authors gratefully acknowledge the financial support of the Leverhulme Trust, Research Project Grant number RPG-2021-092. The authors thank Leron Borsten and Charles Strickland-Constable for helpful discussions. CY would like to thank Alexander Schenkel and James Waldron for helpful discussions.
\end{ack}

\section{The raviolo vacuum module in cochain complexes}\label{sec: the raviolo vacuum module in cochain complexes}
\subsection{The formal raviolo}\label{the-formal-raviolo}
The \dfn{formal raviolo},
\be \Rav \coloneq D \sqcup_{D^\times} D, \nn\ee 
is the $\CC$-scheme obtained by gluing two copies of the formal disc $D = \Spec\CC[[z]]$ along their common copy of the formal punctured disc $D^\times = \Spec \CC((z))$.\footnote{$\Rav=D \sqcup_{D^\times} D$ is an infinitesimal analogue of the affine-line-with-a-doubled-origin, $\AA_\CC^1 \sqcup_{\Spec\CC[z,z^{-1}]} \AA_\CC^1$, which is usually pictured like this 
\be
\begin{tikzpicture}[yscale=0.5]
\draw (-4, 0.1) -- (-2,0.1) .. controls (0,0.1) and (-1,1) .. 
(0,1) node[fill,circle,inner sep=.5mm,label=above:$0$] (0) {} 
.. controls (1,1) and (0,0.1) .. (2,0.1) -- (4,0.1);
\draw (-4, -0.1) -- (-2,-0.1) .. controls (0,-0.1) and (-1,-1) .. 
(0,-1) node[fill,circle,inner sep=.5mm,label=below:$0'$] (0) {}  
.. controls (1,-1) and (0,-0.1) .. (2,-0.1) -- (4,-0.1);
\end{tikzpicture}.
\nn\ee
Both are prototypical examples of \emph{nonseparated} schemes obtained by gluing; see for example \cite[I-44]{EisenbudHarris}. Note that $D^\times$ is open in $D$.}

For our purposes, it is best to visualize the affine line $\AA_\CC^1 = \Spec \CC[z]$ as a copy of the complex \emph{plane}. The formal disc, formal punctured disc, and formal raviolo at some given closed point $a\in \CC$ may be pictured as in \cref{fig: discs and rav}. 

\subsection{Functions on the formal raviolo}\label{sec: functions on the formal raviolo}
By definition, then, the formal raviolo $\Rav$ is the pushout in the category of $\CC$-schemes
\be \begin{tikzcd} D^\times \rar\dar & D \dar \\ D \rar & \Rav \end{tikzcd} \nn\ee
or, equivalently, the coequalizer in $\CC$-schemes
\be 
\Rav = \colim\left( \begin{tikzcd} D^\times \rar[shift left]\rar[shift right] & D\sqcup D
\end{tikzcd} \right).  
\nn\ee
%Here we can think of $\begin{tikzcd} D^\times \rar[shift left]\rar[shift right] & D\sqcup D \end{tikzcd}$ as a semisimplicial object in affine $\CC$-schemes. 
This latter is a useful way to think of $\Rav$ because it presents it explicitly as the colimit of a diagram corresponding to a semisimplicial object in affine schemes. (The notion of a semisimplicial object is recalled in \cref{sec: review of thom-sullivan}.)
Namely it is the colimit of the Čech nerve $\Cech(\mathscr U) = \left( \begin{tikzcd} U_1\cap_\Rav U_2 \rar[shift left]\rar[shift right] & U_1\sqcup U_2 \end{tikzcd}\right)$ of the open cover $\mathscr U= \{U_1,U_2\}$ of $\Rav$ by two copies of the formal disc $U_1 \cong U_2 \cong D$ whose intersection in $\Rav$ is, by definition, a copy of the punctured disc, $U_1\cap_\Rav U_2 \cong D^\times$. 
On applying the global sections functor $\Gamma(-,\O)$, one obtains the semi\emph{co}simplicial object $\Gamma(\Cech(\mathscr U),\mc O)$ in commutative algebras. 
Commutative algebras embed in differential graded (dg) commutative algebras. The derived global sections of the structure sheaf $\O$ on $\Rav$ are then given, by definition, by taking homotopy limit in dg commutative algebras,
\begin{align} R\Gamma(\Rav,\mc O) &\coloneq \holim_{\mathbf{dgCAlg}_\CC} \Gamma(\Cech(\mathscr U),\mc O) \nn\\&= \holim_{\mathbf{dgCAlg}_\CC}\left(\CC((z)) \leftleftarrows \CC[[z]]\times\CC[[z]]\right).\nn\end{align}
As we shall recall in more detail in \cref{sec: derived sections and the Thom-Sullivan functor} and \cref{sec: review of thom-sullivan}, the Thom-Sullivan construction provides a means of computing such homotopy limits.  
Namely, we let \(\CC\sql z\sqr\) denote the dg commutative algebra
\begin{align} \CC\sql z \sqr
      &\coloneq \bigl\{\omega \in \CC((z)) \ox \CC[v,\dd v]: \omega|_{v=0}\in \CC[[z]]\text{ and } \omega|_{v=1}\in \CC[[z]] \bigr\}.\nn\end{align} 
and we then have
\[\CC\sql z\sqr \simeq R\Gamma(\Rav,\mc O).\]
This will be our model, in dg commutative algebras, of the derived global sections of the structure sheaf of $\Rav$.\footnote{Any other model of $R\Gamma(\Rav,\O)$ will be related to this one by a zig-zag of quasi-isomorphisms. In particular this should be true of (the local analogue of) the dg commutative algebra $\A$ of \cite[\S1.2]{GarneWilliRavioloVertexAlgebras2023}.} 

Informally one should think of $\CC\sql z\sqr$ as the ``functions on the formal raviolo $\Rav$'', just as the Laurent series $\CC((z))$ are the ``functions on the formal punctured disc $D^\times$''. The fact that $\CC\sql z\sqr$ is nontrivially differential graded is a reflection of the fact that $\Rav$ is not an affine scheme.
In particular the cohomology of $\CC\sql z \sqr$ computes the sheaf cohomology of the structure sheaf $\O$ of $\Rav$. As a graded vector space, this cohomology is given by
\be H^\bul(\CC\sql z\sqr) \cong_{\mathbf{\grVect}_\CC} \begin{cases}\CC[[z]] & \bul=0\\  
 z^{-1} \CC[z^{-1}] & \bul = 1 \\ 0 & \text{otherwise.} 
                            \end{cases}\nn\ee
(The classes in degree one have representatives in $z^{-1} \CC[z^{-1}] \dd v$. Indeed, such one-forms are closed, obviously, but fail to be exact in $\CC\sql z \sqr$ because of the boundary conditions. For example the would-be primitive $z^{-1} v$ is not in $\CC \sql z \sqr$ since it is not regular in $z$ when pulled back to $v=1$.)

The cohomology $H^\bul(\CC\sql z \sqr)$ comes with the structure of a graded commutative algebra. It is isomorphic, as a graded commutative algebra, to the algebra called \(\mc K\), or \(\CC\langle\langle z \rangle\rangle\), in \cite{GarneWilliRavioloVertexAlgebras2023}:
\be  H^\bul(\CC\sql z \sqr) \cong_{\catname{grAlg}_\CC} \mc K \equiv \CC\langle\langle z\rangle\rangle.\nn\ee

The cohomology has, however, \emph{more} structure than that of a graded commutative algebra.
Indeed, $H^\bul(\CC\sql z \sqr)$  gets the structure of a \dfn{minimal $C_\8$-algebra}, coming from homotopy transfer of the dg commutative algebra structure on $\CC\sql z\sqr\simeq R\Gamma(\Rav,\mc O)$ itself. (See e.g. \cite{LodayValleAlgebraicOperads2012} for a discussion of homotopy transfer, and specifically \cite{ChengGetzler} and \cite[\S13.1.9]{LodayValleAlgebraicOperads2012} for the dg commutative case.) In practice, what that means is that $H^\bul(\CC\sql z \sqr)$ is endowed with a family $(c_k)_{k\geq 2}$ of \dfn{higher products}, the first of which, $c_2$ is the binary product of the graded commutative algebra structure. 
%We give an example of a non-trivial higher product in \cref{a higher product} below. 
One way to keep track of this extra structure is to work in dg commutative algebras rather than passing to their cohomologies.

\subsection{Splitting}\label{sec: splitting}
Let us define the dg commutative algebras
\begin{align} \CC\sql z\sqr_- &\coloneq \left\{\omega \in z^{-1} \CC[z^{-1}]\ox \CC[v,\dd v]: \omega|_{v=0} = 0 \text{ and } \omega|_{v=1} = 0 \right\} \nn\\
\CC\sql z\sqr_+ &\coloneq \CC[[z]] \ox \CC[v,\dd v].\nn\end{align}
Then evidently there are maps of dg commutative algebras
\[ \CC\sql z\sqr_- \into \CC\sql z \sqr \hookleftarrow \CC\sql z\sqr_+ \]
such that at the level of dg vector spaces
\[ \CC\sql z\sqr = \CC\sql z\sqr_- \oplus \CC\sql z\sqr_+. \]
Moreover there are strong deformation retracts of dg vector spaces
\footnote{The map $\iota$ is given by the embedding of $\CC[[w-z_k]] = \CC[[w-z_k]]\ox \CC \xrightarrow{\id \ox 1} \CC[[w-x_k]] \ox \CC[v,\dd v]=\CC\sql w-z_k\sqr_+$, i.e. as ``as constant 0-forms in the $v$ direction''.  The map $(-)|_{\half}$ is given by pulling back forms to (say) $v=\half$. A suitable homotopy $h$ is given by ``$h(\omega)(v) = \int_{\half}^v \omega$'', by which we mean, more precisely, the following: we have $\omega = f(v) + F(v) \dd v$ for some \(f(v), F(v) \in \CC[[w-z_k]] \ox \CC[v]\), and we define $h(f(v) + F(v) \dd v) \coloneq \int_{0}^v  F(v') dv'$.}
\[ \begin{tikzcd}
  \CC\sql z\sqr_+ \arrow[r, shift left=0.75ex, "(-)|_{\half}"] \arrow[loop left,"h"] & \CC[[z]] \cong H^0(\CC\sql z\sqr) \arrow[l, shift left=0.75ex, "\iota"]
  \end{tikzcd}\]
and
\footnote{A suitable homotopy $k$ is given by is given by $k(f(v) + F(v) \dd v) \coloneq \int_{0}^v  F(v') dv' - v \int_0^1 F(v')dv'$. Cf. e.g. \cite[Prop. 16]{AlfonYoungHigherCurrentAlgebras2023a}.}
\[ \begin{tikzcd}
  \CC\sql z\sqr_- \arrow[r, shift left=0.75ex, "\int_0^1"] \arrow[loop left,"k"] & z^{-1}\CC[z^{-1}][1]  \cong H^1(\CC\sql z\sqr)[1]. \arrow[l, shift left=0.75ex, "(-)\dd v"]
  \end{tikzcd}\]

\subsection{Vacuum module}\label{sec: vacuum-module}
Let us now pick a finite-dimensional simple Lie algebra \(\g\) over \(\CC\), for example \(\mf{sl}_2\). We get the dg Lie algebra, i.e. the Lie algebra in dg vector spaces over $\CC$, given by
\[\g \ox \CC\sql z \sqr.\] 
It is the \dfn{raviolo loop algebra}, the raviolo analogue of the usual loop algebra $\g \ox \CC((z))$. It has dg Lie subalgebras \(\g \ox \CC\sql z\sqr_-\) and \(\g \ox \CC\sql z\sqr_+\). We shall think of elements of $\g\ox\CC\sql z\sqr_-$ as \dfn{lowering operators} or \dfn{negative modes}, and of elements of $\g \ox\CC\sql z\sqr_+$ as \dfn{raising operators} or \dfn{positive modes}.
The PBW theorem holds for dg Lie algebras and hence, in view of the splitting above, we have that
\[ U( \g \ox \CC\sql z\sqr) \cong U(\g \ox \CC\sql z \sqr_-) \ox U(\g \ox \CC\sql z \sqr_+) \] 
as dg vector spaces and moreover as $U(\g \ox \CC\sql z \sqr_-), U(\g \ox \CC\sql z \sqr_+)$-bimodules. 
In particular $U( \g \ox \CC\sql z\sqr)$ is free as a  $U(\g \ox \CC\sql z \sqr_-), U(\g \ox \CC\sql z \sqr_+)$-bimodule. 

Let \(\mathscr V\) denote the module over \(\g \ox \CC\sql z\sqr\) induced
\footnote{One can check that %In \cref{sec: V respects quasi-isomorphisms} we check that 
$\mathscr V$ models the derived tensor product $U(\g \ox \CC\sql z \sqr) \lox_{U(\g \ox \CC\sql z\sqr_+)} \CC \vac$ where \[\lox: D({}_{U(\g \ox \CC\sql z \sqr)}\Mod_{U(\g \ox \CC\sql z\sqr_+)}) \times D({}_{U(\g \ox \CC\sql z \sqr_+)}\Mod) \to D({}_{U(\g \ox \CC\sql z \sqr)}\Mod).\] %Note that there is no reason to think it is cofibrant as a left $U(\g \ox \CC\sql z \sqr)$-module.
} 
from the trivial one-dimensional module \(\CC\vac\) over \(\g \ox \CC\sql z \sqr_+\): \[ \mathscr V \coloneq U(\g \ox \CC\sql z \sqr) \ox_{U(\g \ox \CC\sql z\sqr_+)} \CC \vac .\]

Following the usual convention for vertex algebras, we call vectors in \(\mathscr V\) \dfn{states}. The representation $\mathscr V$ is the \dfn{raviolo vacuum Verma module (at level zero)}. 

By the PBW decomposition above, there is an isomorphism
\[ \mathscr V \cong  U(\g \ox \CC\sql z \sqr_-) \ox_\CC \CC \vac\] of left \(U(\g \ox \CC\sql z \sqr_-)\)-modules in dg vector spaces.

\subsection{State-field map}\label{sec: raviolo state-field map}
Let \(\mathscr V((x))\) denote the dg vector space of formal Laurent series in $x$ with coefficients in \(\mathscr V\), which one thinks of as a completion of the tensor product \(\mathscr V\ox \CC((x))\). Let \(\mathscr V\sql x \sqr\) denote the following dg vector space, \[\mathscr V\sql x\sqr
      \coloneq \bigl\{\omega \in  \mathscr V((x)) \ox \CC[u,\dd u]: \omega|_{u=0}\in \mathscr V[[x]]\text{ and } \omega|_{u=1}\in \mathscr V[[x]] \bigr\},\]
which we similarly think as a completion of the tensor product \(\mathscr V \ox \CC\sql x \sqr\).
Now we define a \dfn{state-field map}, namely a (degree-zero) map of dg vector spaces\footnote{Here by $\Hom_{\dgVect_\CC}$ we mean the internal Hom in dg vector spaces, just as in the usual definition of the state-field map -- cf. \cref{sec: vacuum verma module in complex dimension one} -- $\Hom_{\Vect_\CC}$ really means the internal Hom in vector spaces.} 
\[ Y(-;x) : \mathscr V \to \Hom_{\mathbf{dgVect}_\CC}(\mathscr V, \mathscr V\sql x\sqr).\]
We do so recursively, as follows. First, we set \[Y(\vac; x) \coloneq \id_\mathscr V.\]
Then for all homogeneous \(X\in \g \ox \CC\sql x\sqr_-\) and all homogeneous states \(B\in \mathscr V\) we define 
%\[Y\left(X B ; x \right) \coloneq Y_+\left( X\vac; x \right)  Y\left( B ; x \right) + (-1)^{|X||B|} Y\left( B ; x \right) Y_-\left(X \vac; x \right) ,\] 
%where it remains to say what \(Y_+(X\vac,x)\) and \(Y_-(X\vac,x)\) are. 
\[Y\left(X B ; x \right) \coloneq X_+(x)  Y\left( B ; x \right) + (-1)^{|X||B|} Y\left( B ; x \right) X_-(x)  ,\] 
where it remains to say what \(X_+(x)\) and \(X_-(x)\) are. 
To do so, it is enough to define, for all \(a\in \g\) and \(p(v, \dd v) \in \CC[v,\dd v]\), first
\begin{align} 
  \left( a \ox \frac{p(v,\dd v)}{z} \right)_+(x)  &\coloneq \sum_{k=0}^\8 \left( a\ox \frac{p(v,\dd v)}{z^{k+1}} \right) x^k\nn\\
       \left( a \ox \frac{p(v,\dd v)}{z} \right)_-(x)  &\coloneq \sum_{k=0}^{\8} \left( a\ox z^k \right) \frac{p(1-u,-\dd u)}{x^{k+1}}
\nn\end{align} 
and then, for any \(f(z) \in \CC((z))\),
\[\left( a \ox \frac{\del}{\del z} f(z) p(v,\dd v) \right)_\pm(x) \coloneq \frac{\del}{\del x} \left( a \ox f(z) p(v,\dd v) \right)_\pm(x) .\]
More explicitly, one has the following, by induction. 
\begin{lem}[Explicit formula for the state-field map]\label{lem: state-field recursion}
Given a collection of homogeneous lowering operators $X^i \in \g \ox \CC\sql z\sqr_-$, $i=1,\dots,n$, we have 
\begin{align}
  &Y_\Rav( X^1 \dots X^n \vac,x) = \nn\\
  &\sum_{m=0}^n   \sum_{(\mu,\nu) \in \Unshf m n} 
  (-1)^{n-m + \chi(|X^1|,\dots,|X^n|,\mu,\nu)}
  X^{\mu_1}_+(x) \dots X^{\mu_m}_+(x) X^{\nu_{n-m}}_-(x) \dots X^{\nu_1}_-(x) .\nn
\end{align}
The sum is over unshuffles, i.e. permutations $(\mu_1,\dots,\mu_m,\nu_1,\dots,\nu_{n-m})$ of $(1,\dots,n)$ such that $\mu_1<\dots<\mu_m$ and $\nu_1<\dots<\nu_{n-m}$, and $(-1)^{\chi(|X^1|,\dots,|X^n|,\mu,\nu)}$ is the Koszul sign of an unshuffle of the $X^i$, defined such that
\( (-1)^{\chi(|X^1|,\dots,|X^n|,\mu,\nu)} X^1 \dotsm X^n =  X^{\mu_1} \dotsm X^{\mu_m}
   X^{\nu_{n-m}} \dotsm X^{\nu_1} \)
in the symmetric algebra $\Sym(\g \ox \CC\sql w-z_N\sqr_-)$.  
\qed\end{lem}

%Compare \cite[Prop 3.1.21]{LLbook}.
Some remarks are called for about this definition. 

First, one should compare it to that of the usual state-field map for the vacuum Verma module $\VV$ over the loop algebra \(\g \ox \CC((x))\). (Cf. \cref{sec: vacuum verma module in complex dimension one} below.) In some informal but intuitively helpful sense, the former collapses to the latter if one ignores all of the factors $p(v,\dd v)$. 

To understand the role of the polynomial differential form $p(v,\dd v)$, let us consider in turn the examples \[p(v,\dd v) = \dd v\qquad\text{and}\qquad p(v,\dd v) = v(1-v).\] 
The lowering operators 
\[a \ox \frac{\dd v}{z^k}, \qquad a\in \g, \quad k\geq 1,\] 
represent non-trivial cohomology classes in degree one. What we call $\dd v/z^k$ corresponds to what \cite{GarneWilliRavioloVertexAlgebras2023} call $\Omega^{k-1} \propto \lambda^{k-1} \omega$. On such cohomology classes, our definition here coincides with that of \cite[\S2.3]{GarneWilliRavioloVertexAlgebras2023}. (We wrote only a special case of the normal-ordered product above. Cf. \cite[Defn. 2.1.3]{GarneWilliRavioloVertexAlgebras2023}.)
In addition to those lowering operators, there are also lowering operators in degree zero, for example 
\[ a \ox \frac{v(1-v)}{z^k}, \qquad a\in \g,\quad  k\geq 1,\] 
which are neither closed nor exact. This is the first instance of what we meant above by working in dg vector spaces rather than their cohomologies. %(To stress the point: vectors not in the cohomology nevertheless generically contribute to the higher products; cf. the example in \cref{a higher product} below.) 

In view of \cite{GarneWilliRavioloVertexAlgebras2023}, we expect that this definition of the state-field map will make the raviolo vacuum module $\mathscr V$ above into an example of what should probably be called a raviolo vertex algebra in dg vector spaces.  It should be possible to spell out suitable axioms for such a structure, following \cite{GarneWilliRavioloVertexAlgebras2023} in the case of raviolo vertex algebras in graded vector spaces and standard references for vertex algebras, \cite{KacVertexAlgBook,FrenkelBenZvi,LepowLiIntroductionVertexOperator2004}.\footnote{Note also that (standard) vertex algebras internal to the category of dg vector spaces have been studied in \cite{CaradJiangLinDifferentialGradedVertex2023,CaradJiangLinDifferentialGradedVertex2023a}, with rather different motivations. Vertex algebras internal to the category of $\ZZ/2\ZZ$-graded vector spaces, i.e. super vertex algebras, are ubiquitous in the literature; see (but also contrast, because the reference deals with supersymmetric vertex algebras in different sense)\cite{HeluaKacSupersymmetricVertexAlgebras2007}. } 

However, in the present paper we want to do something slightly different. 
Let us adopt the standard perspective that the reason the usual vertex algebra axioms (locality, Borcherds identity, etc) are the way they are is that they formalize the behaviour of what physicists would call operator product expansions (OPEs) in chiral conformal field theories (CFTs). 
More precisely, they capture the limiting behaviour of conformal blocks as two or more insertion points, associated to copies of the vacuum module, are brought close together. 
That relationship between vertex algebras and conformal blocks is known to hold 
for algebraic curves in very great generality -- see \cite{FrenkelBenZvi} and references therein, following especially \cite{TsuchUenotYamadConformalFieldTheory}. But in particular it holds in the prototypical setting of conformal blocks in genus zero, i.e. on the complex plane or the Riemann sphere. 

The crucial point, for us, is that in that latter genus zero setting it is well known how to define conformal blocks \emph{without reference} to vertex algebras. Namely, conformal blocks are defined as the duals of rational coinvariants, as we are about to recall in detail, following \cite{FeigiFrenkResheGaudinModelBethe1994} and \cite[\S13.3]{FrenkelBenZvi}.\footnote{This is what we mean by ``conformal blocks''. The term has a number of closely related meanings in the mathematics and theoretical physics literature. See for example the discussion in \cite[\S8]{CostePaqueCelestialHolographyMeets2022}.} 

Consequently, if the axioms of vertex algebras were mysteriously lost, one principled way to \emph{recover} them would be by studying the limiting behaviour of such rational coinvariants. 

The goal of the present paper is to establish the analogous relationship between vertex algebras and coinvariants in the raviolo case. Namely, we shall define a  notion of coinvariants in the raviolo case, and then we shall show that the state-field map we defined above for the raviolo vacuum module emerges naturally from the behaviour of these raviolo coinvariants in the limit in which marked points collide.

\section{Rational coinvariants and conformal blocks}\label{sec: usual rational conformal blocks}
In this section we review the standard definition of rational coinvariants/conformal blocks on the complex plane with punctures. We focus exclusively on conformal blocks associated to an untwisted affine Kac-Moody algebra; that is, in physics language, the chiral sector of a WZW model. Moreover for simplicity we consider only the case of level zero, i.e. we shall work with loop algebras and not their central extensions.\footnote{It is worth recalling that level zero is in certain important senses a generic value. The non-generic value is the \emph{critical level}, $k=-h^\vee$ in standard normalizations, at which for example the usual Sugawara conformal vector of the vacuum Verma module becomes singular.} 

\subsection{Fixed punctures}\label{sec: fixed punctures in one complex dimension}
In this subsection we work over the complex numbers $\CC$. 

Let $a_1,\dots,a_N \in \CC$ be pairwise distinct complex numbers. We think of them as \emph{marked points} or \emph{punctures} in the complex plane.

Let $\CC[w,(w-a_i)^{-1}]'_{1\leq i\leq N}$ denote the (non-unital) commutative $\CC$-algebra of rational expressions in $w$ singular at most at the points $a_1,\dots,a_N$ and vanishing as $w\to \8$. It is a subalgebra of the unital commutative $\CC$-algebra 
\begin{equation}  \Gamma(\AA_\CC^1\setminus \{a_1,\dots,a_N\},\O) = \CC[w,(w-a_i)^{-1}]_{1\leq i\leq N} \nn\end{equation} 
of sections of the structure sheaf $\O$ of the affine line $\AA_\CC^1 = \Spec \CC[w]$ over the Zariski open subset $\AA_\CC^1\setminus \{a_1,\dots,a_N\}$. 

For each puncture $a_i$, $1\leq i\leq N$, we have the algebra of formal series and of formal Laurent series, 
\begin{equation}  \CC[[w-a_i]] = \Gamma(\Disc_1(a_i),\O)\qquad\text{and}\quad\CC((w-a_i)) = \Gamma(\Discp_1(a_i),\O) \label{def: disc1}\end{equation}which are to be thought of as the algebras of regular functions on, respectively, the formal disc $\Disc_1(a_i) = \Spec \CC[[w-a_i]]$ and the formal punctured disc $\Discp_1(a_i) = \Spec\CC((w-a_i)) = \Disc_1(a_i)\setminus\ol{a_i}$ at the closed point $a_i$ of $\AA_\CC^1$. 

There are embeddings of commutative $\CC$-algebras
\begin{subequations}
\label{lrembeddings}
\begin{equation} \CC[w,(w-a_i)^{-1}]'_{1\leq i\leq N} \hookrightarrow \bigoplus_{i=1}^N \CC((w-a_i)) \hookleftarrow \bigoplus_{i=1}^N \CC[[w-a_i]]  \end{equation}
-- on the left, by Laurent-expanding at each of the marked points; on the right, by the canonical embedding summand by summand -- such that, at the level of vector spaces, there is an isomorphism
\begin{equation} \bigoplus_{i=1}^N \CC((w-a_i)) \cong_\CC  \CC[w,(w-a_i)^{-1}]'_{1\leq i\leq N} \oplus \bigoplus_{i=1}^N \CC[[w-a_i]]  . \end{equation}
\end{subequations}
Let us now pick a simple Lie algebra $\g$ over $\CC$. We get Lie algebras over $\CC$
\begin{equation} \b \coloneq \bigoplus_{i=1}^N \g \ox \CC((w-a_i)) \nn\end{equation}
\begin{equation} \b_+ \coloneq \bigoplus_{i=1}^N \g \ox \CC[[w-a_i]],\qquad \b_- \coloneq \g \ox \CC[w,(w-a_i)^{-1}]'_{1\leq i\leq N} \nn\end{equation} 
and embeddings of Lie algebras over $\CC$
\begin{subequations}\label{manin a}
\begin{equation} \b_- \hookrightarrow  \b \hookleftarrow \b_+ \end{equation}
which again give rise to an isomorphism of the underlying vector spaces,
\begin{equation} \b \cong_\CC \b_- \oplus \b_+ .\end{equation}
\end{subequations}

Let $M_i$, $1\leq i\leq N$ be $\g$-modules in the category of $\CC$-vector spaces. 
We make each $M_i$ into a module over the Lie algebra $\g\ox \CC[[w-a_i]]$ by declaring $X \ox 1$ acts as $X$ and $X\ox (w-a_i)^k$ acts as $0$ for all $k\geq 0$ and all $X\in \g$. In other words, we pull back $M_i$ along the map of Lie algebras $\g\ox \CC[[w-a_i]]\to \g\ox \CC[[w-a_i]]\big/ \g\ox (w-a_i)\CC[[w-a_i]] \cong \g $. We may then construct the induced $\g \ox \CC((w-a_i))$ module
\begin{equation} \MM_i \coloneq \Ind_{\g \ox \CC[[w-a_i]]}^{\g \ox \CC((w-a_i))} M_i \coloneq U(\g \ox \CC((w-a_i))) \ox_{U(\g \ox \CC[[w-a_i]])} M_i. \label{def: MMi} \end{equation}
Equivalently, one sees that
\begin{equation} \MM \coloneq \bigotimes_{i=1}^N \MM_i = \Ind_{\b_+}^\b M,\quad\text{where}\quad M\coloneq \bigotimes_{i=1}^N M_i \nn\end{equation}
On pulling back by the embedding $\b_-\into \b$, $\MM$ is, in particular, a module over $\b_-$. 

The space of \dfn{rational coinvariants (at level zero)} associated to these data $a_1,\dots,a_N$; $\g$; $M_1,\dots,M_N$ is then by definition %the space of coinvariants
\begin{equation} \coinv(\g;a_1,\dots,a_N;M_1,\dots,M_N) \coloneq \MM \big/ \b_- \coloneq \MM \big/(\b_- \on \MM) \cong_\CC \CC \ox_{U(\b_-)} \MM \nn\end{equation}
and the space of \dfn{rational conformal blocks (at level zero)} is by definition the dual space
\begin{equation} \Hom_\CC(\MM \big/\b_-,\CC) \cong_\CC \Hom_{\Mod_\CC(\b_-)}(\MM,\CC). \nn\end{equation}
Here we make $\CC$ into the trivial $\b_-$ module. Equivalently it is the $U(\b_-)$-module obtained by pulling back by the counit map $U(\b_-) \to \CC$.

It follows from \cref{manin a} together with the PBW theorem that there is an isomorphism $U(\b) \cong U(\b_-) \ox_\CC U(\b_+)$ of $(U(\b_-),U(\b_+))$-bimodules. Consequently the space of coinvariants is canonically isomorphic to underlying vector space of the $\g$-module $M=M_1\ox\dots\ox M_N$ from which we began:
\begin{align} \MM \big/ \b_- = \CC \ox_{U(\b_-)} \MM 
&\cong_\CC \CC \ox_{U(\b_-)} U(\b) \ox_{U(\b_+)} M \nn\\ 
&\cong_\CC \CC \ox_{U(\b_-)} U(\b_-) \ox_\CC U(\b_+) \ox_{U(\b_+)} M \nn\\&\cong_\CC \CC \ox_\CC M = M. \nn
\end{align} 
The richness of the space of coinvariants really emerges when one allows the marked points $a_1,\dots,a_N$ to vary.\footnote{One sense in which this is true is that the rational coinvariants/conformal blocks obey the celebrated KZ equations. It would be interesting to investigate whether the coinvariants we introduce below in the ravioli setting obey some analogous equations.} 
We turn to this now.

\subsection{Movable punctures and configuration space}\label{sec: movable punctures in one complex dimension}
Our fixed set of distinct marked points $a_1,\dots,a_N \in \CC$ from the previous subsection is now to be thought of just one choice of closed point of the \dfn{configuration space} 
\begin{equation} \Conf_N \coloneq \AA_\CC^N \setminus \bigcup\ijNp \ol{(z_i=z_j)} \label{def: CN}\end{equation}
obtained by starting with the affine scheme 
$\AA_\CC^N = \Spec \CC[z_1,\dots,z_N]$
and removing (the Zariski closures $\ol{(z_i=z_j)}$ of the generalized points $(z_i=z_j)$ of) all the diagonal hyperplanes. 

Following the approach of \cite[\S13]{FrenkelBenZvi}, one can think that going from fixed to movable marked points is a matter of changing the ground ring from $\CC$ to the $\CC$-algebra $\B_N \coloneq \Gamma(\Conf_N, \O)$ of regular functions on configuration space. Since $\Conf_N$ is the complement in $\AA_\CC^N$ of the zero locus of the function 
\begin{equation} \Delta_N \coloneq \prod\ijNp (z_i-z_j) \nn\end{equation} 
we have by definition -- see e.g. \cite[\S I]{EisenbudHarris} -- that
\begin{subequations}\label{def: BN}
\begin{equation} \B_N = \CC[z_1,\dots,z_N][\Delta_N^{-1}] \end{equation} 
is the localization of $\CC[z_1,\dots,z_N] = \Gamma(\AA_\CC^N,\O)$ obtained by adjoining an inverse to $\Delta_N$, and 
\begin{equation} \Conf_N = \Spec \B_N\nn\end{equation} 
is again an affine scheme. Of course, once we can invert $\Delta_N$, we can invert any $(z_i-z_j)$, so that, in  more suggestive notation, 
\begin{equation} \B_N = \CC[z_i,(z_i-z_j)^{-1}]\ijN. \end{equation}
\end{subequations}

Now we should ask what, in this setting, the analogues of the embeddings of $\CC$-algebras in \cref{lrembeddings} should be. First, in place of $\CC[w,(w-a_i)^{-1}]_{1\leq i\leq N} = \Gamma(\AA_\CC^1\setminus\{a_1,\dots,a_N\},\O)$ we should consider
\begin{equation} \B_{N+1} = \CC[z_1,\dots, z_N,w]\left[\left(\Delta_N \prod_{i=1}^N (w-z_i)\right)^{-1}\right] = \Gamma(\Conf_{N+1},\O). \label{def: BNplus1}\end{equation}
This is a $\B_N$-algebra in the obvious way. Let $\B_{N+1}'$ denote the (non-unital) subalgebra consisting of those functions that vanish as $w\to \8$. 

Next, we want the analogue of the disc $\Disc_1(a_i) = \Spec\CC[[w-a_i]]$ and the punctured disc $\Discp_1(a_i) = \Spec\CC((w-a_i))$ near a closed point $a_i\in \AA_\CC^1$. Recall that $\CC[[w-a_i]]$ is the completion of $\CC[w]$ with respect to the ideal $(w-a_i)\CC[w]$, and $\CC((w-a_i))$ is then the localization of $\CC[[w-a_i]]$ obtained by adjoining an inverse to $w-a_i$. 
Here, since the ground ring is now $\B_N$, we consider the completion of $\B_{N}[w]$ with respect to $(w-z_i) \B_{N}[w]$, i.e. the ring $\B_N[[w-z_i]]$, and then its localization $\B_N((w-z_i))$. 

We arrive at the following analogue of \cref{lrembeddings}: there are embeddings of commutative algebras in $\B_N$-modules
\begin{equation} \B_{N+1}' \hookrightarrow \bigoplus_{i=1}^N \B_N((w-z_i)) \hookleftarrow \bigoplus_{i=1}^N \B_N[[w-z_i]] \label{movablelrembeddings}\end{equation}
such that, at the level of $\B_N$-modules, there is an isomorphism
\begin{equation} \bigoplus_{i=1}^N \B_N((w-z_i)) \cong_{\B_N} \B_{N+1}' \oplus \bigoplus_{i=1}^N \B_N[[w-z_i]]  . \label{movabledecomp}\end{equation}
Let $\g$ be a simple Lie algebra over $\CC$ as before. By extension of scalars we obtain the Lie algebra over $\B_N$, i.e. the Lie algebra in $\B_N$-modules, given by
\begin{equation} \prescript{}\B \g \coloneq \B_N\ox \g.\nn\end{equation} 

We have also Lie algebras in $\B_N$-modules given by
\begin{gather*}
{\prescript{}{\B}{\b}} \coloneq \bigoplus_{i=1}^N \g \ox \B_N((w-z_i)) \\
{\prescript{}{\B}{\b}_+} \coloneq \bigoplus_{i=1}^N \g \ox \B_N[[w-z_i]],\qquad {\prescript{}{\B}{\b}_-} \coloneq \g \ox \B_{N+1}'
\end{gather*} 
and embeddings of Lie algebras in $\B_N$-modules
\begin{subequations}\label{manin b}
\begin{equation} {\prescript{}{\B}{\b}_-} \hookrightarrow  {\prescript{}{\B}{\b}} \hookleftarrow {\prescript{}{\B}{\b}_+} \end{equation}
which give rise to an isomorphism of the underlying $\B_N$-modules,
\begin{equation} {\prescript{}{\B}{\b}} \cong {\prescript{}{\B}{\b}_-} \oplus {\prescript{}{\B}{\b}_+} .\end{equation}
\end{subequations}

Let $M_i$, $1\leq i\leq N$ be $\g$-modules as before. By extension of scalars we get $\prescript{}\B \g$-modules $\prescript{}\B M_i\coloneq\B_N \ox M_i$, i.e. $\prescript{}\B \g$-module objects in the category of $\B_N$-modules. 
We make each $\prescript{}\B M_i$ into a module over the Lie algebra $\g\ox \B_N[[w-z_i]]$ by pulling back along the map of Lie algebras $\g\ox \B_N[[w-z_i]]\to \g\ox \B_N[[w-z_i]]\big/ \g\ox (w-z_i)\B_N[[w-z_i]] \cong \prescript{}\B \g $. We may then construct the induced $\g \ox \B_N((w-z_i))$-module
\begin{align} \prescript{}{\B}\MM_i &\coloneq \Ind_{\g \ox \B_N[[w-z_i]]}^{\g \ox \B_N((w-z_i))} \prescript{}\B M_i\nn\\
 &\coloneq U_{\B_N}(\g \ox \B_N((w-z_i))) \ox_{U_{\B_N}(\g \ox \B_N[[w-z_i]])} \prescript{}\B M_i. \label{def: BMMi}
\end{align}
(Here we write $U_{\B_N}(-): \LieAlg(\Mod_{\B_N}) \to \AssAlg(\Mod_{\B_N})$ for the functor whose action on objects is to take a Lie algebra over $\B_N$ to its universal envelope, an associative algebra over $\B_N$.)

The tensor product of these modules,
\begin{equation} \prescript{}{\B}\MM \coloneq \bigotimes_{i=1}^N{}_{\B_N} \prescript{}{\B}\MM_i,\nn\end{equation}
is equivalently the induced module
\begin{equation} \prescript{}{\B}\MM = \Ind_{{\prescript{}{\B}{\b}_+}}^{\prescript{}{\B}{\b}} \prescript{}\B M \coloneq U_{\B_{N}}(\b) \ox_{U_{\B_{N}}({\prescript{}{\B}{\b}_+})} \prescript{}\B M, \quad\text{where}\quad \prescript{}\B M\coloneq \bigotimes_{i=1}^N{}_{\B_N} \prescript{}\B M_i .\nn\end{equation}
On pulling back by the embedding ${\prescript{}{\B}{\b}_-}\into {\prescript{}{\B}{\b}}$, $\prescript{}{\B}\MM$ is, in particular, a module over ${\prescript{}{\B}{\b}_-}$. 

The space of \dfn{rational coinvariants (at level zero)} associated to these data $\g$; $M_1,\dots,M_N$ is then by definition the $\B_N$-module
\begin{align} \coinv(\g;\B_N;M_1,\dots,M_N) \coloneq  \prescript{}{\B}\MM \big/ {\prescript{}{\B}{\b}_-} &\coloneq \prescript{}{\B}\MM \big/({\prescript{}{\B}{\b}_-} \on \prescript{}{\B}\MM) \nn\\ 
&\cong_{\B_N} \B_N \ox_{U_{\B_N}({\prescript{}{\B}{\b}_-})} \prescript{}{\B}\MM \label{BN coinvariants}\end{align}
and the space of \dfn{rational conformal blocks (at level zero)} is by definition the dual
\begin{equation} \Hom_{\B_N}(\prescript{}{\B}\MM \big/ {\prescript{}{\B}{\b}_-},\B_N) \cong_{\B_N} \Hom_{{\prescript{}{\B}{\b}_-}}({\prescript{}{\B}\MM},\B_N). \nn\end{equation}
Here we make $\B_N$ into the trivial ${\prescript{}{\B}{\b}_-}$-module. Equivalently it is the $U_{\B_N}({\prescript{}{\B}{\b}_-})$-module obtained by pulling back by the counit map $U_{\B_N}({\prescript{}{\B}{\b}_-}) \to \B_N$.

Once more, the data of the triple of Lie algebras \cref{manin b} together with the PBW theorem imply that there is an isomorphism, $U_{\B_N}(\b) \cong U_{\B_N}({\prescript{}{\B}{\b}_-}) \ox_{\B_N} U_{\B_N}({\prescript{}{\B}{\b}_+})$, now of $(U_{\B_N}({\prescript{}{\B}{\b}_-}),U_{\B_N}({\prescript{}{\B}{\b}_+}))$-bimodules in $\B_N$-modules, and hence that the space of coinvariants is canonically isomorphic to $\prescript{}\B M = \B_N \ox M$:
\begin{align} \prescript{}{\B}\MM \big/ {\prescript{}{\B}{\b}_-} &= \B_N \ox_{U_{\B_N}({\prescript{}{\B}{\b}_-})} \prescript{}{\B}\MM \nn\\
&\cong_{\B_N} \B_N \ox_{U_{\B_N}({\prescript{}{\B}{\b}_-})} U_{\B_N}(\b) \ox_{U_{\B_N}({\prescript{}{\B}{\b}_+})} \prescript{}\B M \nn\\ 
&\cong_{\B_N} \B_N \ox_{U_{\B_N}({\prescript{}{\B}{\b}_-})} U_{\B_N}({\prescript{}{\B}{\b}_-}) \ox_{\B_N} U_{\B_N}({\prescript{}{\B}{\b}_+}) \ox_{U_{\B_N}({\prescript{}{\B}{\b}_+})} \prescript{}\B M \nn\\ &\cong \B_N \ox_{\B_N} \prescript{}\B M = \prescript{}\B M = \B_N \ox M. \label{floating coinvariants}
\end{align}

\subsection{Taking coinvariants}\label{sec: taking coinvariants in dimension one}
For any $\CC$-point $(a_1,\dots,a_N) \in \Conf_N$ we can apply the evaluation map 
\begin{equation} \ev_{a_1,\dots,a_N} : \B_N \to \CC \nn\end{equation}
to recover the space of coinvariants for this particular choice of fixed marked points, as we had it in \cref{sec: fixed punctures in one complex dimension}. That is, the constructions above are all suitably functorial, so that there is a map
\begin{equation} \ev_{a_1,\dots,a_N} : \coinv(\g;\B_N;M_1,\dots,M_N) \to \coinv(\g;a_1,\dots,a_N;M_1,\dots,M_N). \nn\end{equation}
In checking this, one notes in particular the following lemma.
\begin{lem}\label{lem: MM base change}
For each $i$, $\prescript{}{\B}\MM_i \cong \B_N \ox_\CC \MM_i$. Hence $\prescript{}{\B}\MM \cong \B_N \ox_\CC \MM$. 
\end{lem}
\begin{proof}
Let us define 
\begin{equation} \gm \coloneq (w-z_N)^{-1}\CC[(w-z_N)^{-1}]\quad\text{and}\quad \prescript{}\B \gm \coloneq (w-z_N)^{-1}\B_N[(w-z_N)^{-1}].\nn\end{equation} 
These are Lie algebras over $\CC$ and $\B_N$ respectively. Certainly we have 
\begin{equation}  \prescript{}\B \gm \cong  \B_N \ox_\CC \gm \label{gm isom}\end{equation} 
(since these are just Lie algebras of polynomials) and hence that $U_{\B_N}(\prescript{}\B \gm) \cong \B_N \ox_\CC U(\gm)$. 
We have the ``local'' Lie algebra splitting, i.e. the embeddings of Lie algebras in $\B_N$-modules
\begin{equation} \prescript{}\B \gm  \hookrightarrow \g \ox \B_N((w-z_N)) \hookleftarrow \g \ox \B_N[[w-z_N]] \nn\end{equation}
such that as $\B_N$-modules
\begin{equation} \prescript{}\B \gm \oplus \g \ox \B_N[[w-z_N]] \cong \g \ox \B_N((w-z_N)).\nn\end{equation}
Therefore $\prescript{}{\B}\MM_i$ is free as a module over $U_{\B_N}(\gm)$:
\begin{equation} \prescript{}{\B}\MM_i \cong U_{\B_N}(\prescript{}\B \gm) \ox_{\B_N} \prescript{}\B M_i \nn\end{equation} 
The analogous statements hold for $\gm$ and $\MM_i$. The result follows.
\end{proof}

On the other hand, we can now consider fixing vectors $m_i \in \MM_i$, for $1\leq i\leq N$, while letting the marked points vary. More precisely, the unit map $1: \CC \to \B_N$ induces embeddings (of vector spaces) 
\begin{equation} \MM_i \into \B_N \ox \MM_i \cong \prescript{}{\B}\MM_i;\quad m \mapsto 1 \ox m\label{MMi into BMMi}\end{equation} 
for each $i$, and hence $\MM\into \B_N \ox \MM \cong \prescript{}{\B}\MM$.
We may take the class of the vector $m_1\ox \dots \ox m_N \in \MM \cong 1\ox \MM \into \B_N \ox \MM \cong \prescript{}{\B}\MM$ in the space of coinvariants in \cref{floating coinvariants}. We shall write this class as
\begin{equation} \bigl[ \atp{z_1}{m_1} \ox \dots \ox \atp{z_N}{m_N} \bigr].\nn\end{equation}
It is an element of $\B_N \ox M$, i.e. an $M = \bigotimes_{i=1}^N M_i$-valued rational function of $z_1,\dots,z_N$ singular at most on the diagonals $z_i-z_j$, $1\leq i<j\leq N$. 

We call applying this map of $\B_N$-modules 
\begin{equation}\MM \to \B_N \ox M;\quad m_1 \ox \dots \ox m_N \mapsto \bigl[ \atp{z_1}{m_1} \ox \dots \ox \atp{z_N}{m_N} \bigr],\label{eq: taking coinvariants}\end{equation}
the operation of \dfn{taking coinvariants}. 

%Let us stress that the fact that this is a well-defined thing to do relies on the (obvious) fact that there is a canonical projection map ${}^B\MM \onto \prescript{}{\B}\MM/{\prescript{}{\B}{\b}_-}$, i.e. 
%\begin{equation} \prescript{}{\B}\MM \onto \B_N \ox_{U_{\B_N}({\prescript{}{\B}{\b}_-})} \prescript{}{\B}\MM ; \qquad m \mapsto [m] \coloneq 1 \ox_{U_{\B_N}({\prescript{}{\B}{\b}_-})} m , \nn\end{equation}
%a point which will become more subtle below. 

\subsection{The usual vacuum Verma module and state-field map}\label{sec: vacuum verma module in complex dimension one}
Let now
\begin{equation} \VV \coloneq \Ind_{\g\ox \CC[[s]]}^{\g \ox \CC((s))} \CC\vac \nn\end{equation}
denote the module over the loop algebra $\g \ox \CC((s))$ induced from the trivial one-dimensional module $\CC\vac$ over $\g\ox \CC[[s]]$ generated a vector $\vac$.
This module $\VV$ is called the \dfn{vacuum Verma module (at level zero)}. 
%When we we work over the algebra $\B_N = \Gamma(\Conf_N,\O)$ of functions on the configuration space $\Conf_N$, there is an isomorphism of $\B_N$-modules
%\be \B_N \ox \VV_i \cong_{\B_N} \Ind_{\g \ox \B_N[[w-z_i]]}^{\g\ox \B_N((w-z_i))} \B_N \vac  \label{indisom}\ee
%indeed, as $\B_N$-modules, both are isomorphic to
%\be \B_{N} \ox U(\g\ox (w-z_i)^{-1} \CC[(w-z_i)^{-1}]) \cong_{\B_N} U_{\B_N}(\g \ox (w-z_i)^{-1} \B_N[(w-z_i)^{-1}]). \nn\ee 
%In this sense, the change of base ring from $\CC$ to $\B_N$ commutes with induction, in the usual case of complex dimension one. 

%\subsection{The usual state-field map}
The vacuum Verma module $\VV$ %from \cref{sec: vacuum verma module in complex dimension one} 
comes equipped with a linear map 
\begin{equation} Y(-, s) : \VV \to \Hom_{\Vect_\CC}(\VV, \VV((s))) ,\qquad A \mapsto Y(A, s) = \sum_{n \in \ZZ} A_{(n)} s^{-n-1} ,\label{Ymap}\end{equation}
called the \dfn{state-field map}. Vectors in $\VV$ are called \dfn{states}, and one can think that $Y(-,s)$ sends each state $A\in \VV$ to the formal sum of its \dfn{modes} $A_{(n)} \in \End_{\Vect_\CC}(\VV)$. The state-field map satisfies certain axioms (notably Borcherds identity) which can be found in standard references including  \cite{KacVertexAlgBook,LepowLiIntroductionVertexOperator2004,FrenkelBenZvi} and which make $\VV$ into a \dfn{vertex algebra}. 
From our present perspective the important point is that one way to motivate these axioms is by studying the limit of rational coinvariants as points collide, as we now describe. %shall describe in \cref{prop: Y map} below.

Let us now specialize our discussion of rational coinvariants above to the case in which $M_{N-1} = \CC\vac$ and $M_N = \CC\vac$. 
In that case
\begin{equation} M = \bigotimes_{i=1}^{N-2} M_i \ox \CC \ox \CC = \bigotimes_{i=1}^{N-2} M_i\qquad\text{and}\qquad \MM = \bigotimes_{i=1}^{N-2} \MM_i \ox \VV_{N-1} \ox \VV_N. \nn\end{equation}
Now, by identifying local coordinates $w-z_i$, we may identify each of the Lie algebras $\g\ox \CC((w-z_i))$ with a single copy $\g\ox \CC((s))$,\footnote{This identification is sometimes left implicit, but is an important assumption. Other ways of picking preferred local coordinates and hence identifying copies of $\VV$ at different points are possible and can be important in applications: for example when one wants to go to what physicists would call the cylinder geometry.} 
%\begin{equation} \g\ox \CC((s)) \cong \g\ox \CC((w-z_i)), \nn\end{equation}
and thereby identify their vacuum Verma modules with a single abstract copy of $\VV$:
\begin{equation} \VV\cong \VV_i \coloneq \Ind_{\g\ox \CC[[w-z_i]]}^{\g \ox \CC((w-z_i))} \CC\vac. \nn\end{equation}
Pick vectors $m_i\in \MM_i$ for $1\leq i\leq N-2$, and states $A,B\in \VV$ in this abstract copy of $\VV$. On taking coinvariants, we get the $M$-valued rational function
\begin{align} 
\bigl[ \atp{z_1}{m_1} \ox \dots \ox \atp{z_{N-2}}{m_{N-2}} \ox \atp{z_{N-1}}{B} \ox \atp{z_N}{A} \bigr] \in \B_N \ox M. \label{coinvariant}
\end{align}
%(Recall that $\B_N = \Gamma(\Conf_N,\O)$ is the algebra of rational functions on configuration space $\Conf_N \coloneq \AA_\CC^N \setminus \bigcup\ijN \ol{(z_i=z_j)}$.)

The usual state-field map captures the behaviour of this function in the limit in which the marked point $z_N$ becomes close to marked point $z_{N-1}$, while the points $z_1,\dots,z_{N-1}$, the vectors $m_1,\dots,m_{N-2}$, and the states $A, B$ are all held fixed.

\begin{thm}[Relation of the state-field map $Y$ to rational coinvariants]\label{prop: Y map}
For all $A,B\in \VV$ and $m_i\in \MM_i$, $1\leq i\leq N-2$, we have 
\begin{align} &\iota_{z_{N}\to z_{N-1}}
\bigl[ \atp{z_1}{m_1} \ox \dots \ox \atp{z_{N-2}}{m_{N-2}} \ox \atp{z_{N-1}}{B} \ox \atp{z_N}{A} \bigr] \nn\\
&\quad\qquad
= 
\bigl[ \atp{z_1}{m_1} \ox \dots \ox \atp{z_{N-2}}{m_{N-2}} \ox  Y(A;z_N-z_{N-1}) \!\!\atp{z_{N-1}}{B} \bigr]. 
\nn\end{align}
%as an equality in $\B_{N-1}((z_N-z_{N-1}))\ox M$.
\end{thm}
We give a proof of this very standard fact below, in \cref{sec: proof: Y map}.
Note that while the left-hand side here is manifestly in $\B_{N-1}((z_N-z_{N-1})) \ox M$, since it is the expansion of an element of $\B_N\ox M$, a priori the right-hand side is merely an element of $(\B_{N-1}\ox M)((z_N-z_{N-1}))$: indeed what is meant by the right-hand side is the series obtained by computing the coinvariant in $\coinv(\g;\B_{N-1};M_1,\dots,M_{N-2},\CC) \cong \B_{N-1} \ox M$ order by order in $z_N-z_{N-1}$.

Now let us close this digression on the usual rational conformal blocks associated to  Kac-Moody algebras at level zero, and return to the  raviolo case.

\section{Ravioli configuration space}\label{sec: configuration spaces of ravioli}

%In this section our goal is to do two things: 
%\begin{enumerate}[1.] 
% \item  to introduce a suitable notion of \dfn{ravioli configuration space} $\RavConf_N$, and
% \item to introduce a convenient model of the derived global sections of its structure sheaf
%\end{enumerate}
As we just saw, %in \cref{sec: movable punctures in one complex dimension}, 
the usual rational coinvariants/conformal blocks form a module over the commutative algebra
\begin{align} \B_N &\coloneq \Gamma(\Conf_N,\O) = \CC[z_1,\dots,z_N][\mathrm{Discr}_N^{-1}]\nn\\ &= \CC[z_i,(z_i-z_j)^{-1}]\ijN \nn\end{align} 
of global sections of the structure sheaf on the configuration space of $N$ distinguishable pairwise distinct marked points in the complex plane,
\begin{equation} \Conf_N \coloneq \AA_\CC^N \setminus \bigcup\ijNp \ol{(z_i=z_j)} \nn.\end{equation}

\begin{rem}\label{rem: bundles}
In more geometrical language, the module of conformal blocks is the $\Gamma(\Conf_N,\O)$-module of global sections of a trivial vector bundle over configuration space, whose typical fibre we described in \cref{sec: fixed punctures in one complex dimension}; see \cite[\S13.3]{FrenkelBenZvi}. Cf. also e.g. \cite{VarchenkoKZlectures}, \cite{EFK}.
\end{rem}

In this section our goal is to introduce a suitable notion of \dfn{configuration space in the ravioli setting}, \[\RavConf_N,\] and then to construct a model 
\[ \A_N \simeq R\Gamma(\RavConf_N,O), \]
in dg commutative algebras, of the derived global sections of its structure sheaf.
This is a prelude to defining ravioli analogues of rational coinvariants/conformal blocks, which we shall do in the next section.

In the ravioli setting, we again want to describe configurations of $N$ distinguishable marked points in the complex plane. However, we now want to allow  them to coincide, but with the stipulation that, whenever two points \emph{do} coincide, we want to keep track of which point is ``on top'' of which. 

Formally then, what we shall do is to glue together $N!$ copies of the affine scheme $\AA_\CC^N$ along the complements of the diagonal hyperplanes, as follows. 

Let $S_N$ denote the group of permutations of the set $[1,N] = \{1,\dots,N\}$. We shall identify $S_N$, as a set, with the set of \dfn{total orders} on $[1,N]$, by associating $\sigma\in S_N$ with the total order $\prec_\sigma$ on $[1,N]$ defined by
\[ \sigma(1) \prec_\sigma \sigma(2) \prec_\sigma \dots \prec_\sigma \sigma(N).\]
Let $\Partial_{[1,N]}$ denote the set of all \dfn{partial orders} $\prec$ on the set $[1,N]=\{1,\dots,N\}$. We make $\Partial_{[1,N]}$ itself into a partially ordered set (or, equivalently, a skeletal preorder) in which there is an arrow $\prec \to \prec'$ if and only if $\prec'$ refines $\prec$ in the obvious sense.
Given a partial order \(\prec\in \Partial_{[1,N]}\), let \(\mc O(U_\prec)\) denote the commutative algebra
\[\mc O(U_\prec) \coloneq \CC[z_1,\dots,z_N]\left[\prod_{\substack{i,j\in [1,N]\\ i\neq j, i\not\prec j, j\not\prec i}} % \text{distinct, $\prec$-incomparable} \\ \text{elements of $[1,N]$}}}
\frac{1}{z_i-z_j}\right]\in \CAlg(\Vect_\CC).\]
That is, $\O(U_\prec)$ is the localization of the polynomial algebra $\CC[z_1,\dots,z_N]$ in which $z_i-z_j$ is invertible precisely for those distinct $i$ and $j$ that are $\prec$-incomparable. 
%In particular, if $\prec$ is a total order then $\O(U_\prec) = \CC[z_1,\dots,z_N]$ and $U_\prec$ is a copy of $\AA_\CC^N$. 
If $\prec'$ refines $\prec$ then there is a canonical inclusion of $\CC$-algebras
$\mc O(U_{\prec'}) \to \mc O(U_\prec)$, and hence a map of affine $\CC$-schemes $U_{\prec}\to U_{\prec'}$. This defines a functor\footnote{and hence a diagram in $\CC$-schemes. In fact $\RavConf_N$ as we are about to define it is the colimit in $\CC$-schemes of this diagram, $\RavConf_N=\colim_{\prec\in \Partial_{[1,N]}} U_\prec$.}
\[ U:\Partial_{[1,N]} \to \CAlg(\Vect_\CC)^\op\equiv \mathbf{AffSch}_\CC \into \mathbf{Sch}_\CC\] 

\subsection{Definition of $\RavConf_N$ by gluing Čech data}\label{the-thom-sullivan-functor}
Recall that if we are \emph{given} a $\CC$-scheme $X$, we can by definition always cover it with a collection $\mathscr U = \{U_i\}_{i\in I}$ of affine patches $U_i\in \catname{AffSch}_\CC$ indexed by some totally ordered index set $(I,<)$. The Čech nerve of this cover is the semisimplicial object in $\CC$-schemes given by the diagram
\[ \Cech(\mathscr U) = \left( 
\begin{tikzcd} 
\dotsb \quad %\prod_{\substack{(X,Y,Z) \in I^3\\ X<Y<Z}} \O(U_X \cap U_Y \cap U_Z)  
%\rar[<-,shift left=8pt]\rar[<-,shift left=4pt]\rar[<-]\rar[->,shift right = 4pt]\rar[->,shift right = 8pt]& 
\rar[->,shift left=4pt]\rar[->,shift left=0pt]\rar[->,shift right = 4pt]& 
\bigsqcup_{\substack{i,j\in I\\ i<j}} U_{i,j}  
%\rar[<-shift left=4pt]\rar[<-] \rar[->,shift right=4pt]& 
\rar[->,shift left=2pt]\rar[->,shift right=2pt]& 
\bigsqcup_{i\in I} U_i
\end{tikzcd}\right)\] 
where \[U_{i,j}\coloneq U_i\cap_X U_j\] is the intersection in $X$ of the affine patches $U_i$ and $U_j$. 
(We recall the meaning of semisimplicial objects in a category in \cref{sec: review of thom-sullivan}.)
This is a diagram in $\CC$-schemes whose colimit is the original scheme $X$:
\[ X = \colim \Cech(\mathscr U) .\] 
Moreover we may always find a Leray cover, and the Čech cohomology of $\O$ with respect to such a cover computes the sheaf cohomology of $\O$. This is the case in particular if all the intersections $U_{i,j,\dots,k}$ are themselves affine.
%; note that in that case the Čech nerve $\Cech(\mathscr U)$ is a semisimplicial object in affine schemes.

With this in mind, let us define the \dfn{ravioli configuration space $\RavConf_N$} as follows. It is covered by the collection of $N!$ affine schemes
\[ \mathscr U \coloneq \left\{ U_{\prec_\sigma} \cong \AA_\CC^N: \sigma \in S_N \right\}.\]
We glue these affine patches together as follows. 
Given a collection of partial orders $\prec_1,\dots,\prec_k \in \Partial_{[1,N]}$,
let \[\prec_1\wedge\dots\wedge \prec_k\] denote their finest common coarsening, or in other words their meet, or categorical product, in \(\Partial_{[1,N]}\). 
All meets exist in \(\Partial_{[1,N]}\); at worst the meet may be the initial object, namely the partial order in which no two elements are comparable. 
Given any distinct $\sigma_1,\sigma_2\in S_N$ we \emph{define} the intersection in $\RavConf_N$ of the affine patches $U_{\prec_{\sigma_1}}$ and $U_{\prec_{\sigma_2}}$ to be $U_{\prec_{\sigma_1}\wedge\prec_{\sigma_2}}$:
\begin{equation} U_{\prec_{\sigma_1}}\cap_{\RavConf_N} U_{\prec_{\sigma_2}} \coloneq U_{\prec_{\sigma_1}\wedge\prec_{\sigma_2}},\nn\end{equation}
with the inclusions $U_{\prec_{\sigma_1}\wedge\prec_{\sigma_2}} \into U_{\prec_{\sigma_1}}$ and $U_{\prec_{\sigma_1}\wedge\prec_{\sigma_2}} \into U_{\prec_{\sigma_2}}$ being the canonical inclusions we noted in the definition of the functor $U$ above. This gluing data satisfies the triple overlap condition, and we define $\RavConf_N$ to be the resulting $\CC$-scheme. (See e.g. \cite[\S I.2.4 and Corollary I-14]{EisenbudHarris} for a discussion of the gluing construction.) Indeed, given any $\sigma_1,\dots,\sigma_k \in S_N$, the intersection of $U_{\prec_{\sigma_1}},\dots,U_{\prec_{\sigma_k}}$ in $\RavConf_N$ is then a copy of the affine scheme $U_{\prec_{\sigma_1}\wedge\dots\wedge \prec_{\sigma_k}}$:
\[ U_{\prec_{\sigma_1}}\cap_{\RavConf_N} \dots \cap_{\RavConf_N} U_{\prec_{\sigma_k}} \coloneq U_{\prec_{\sigma_1}\wedge\dots\wedge \prec_{\sigma_k}}.\]
Thus, we pick and fix arbitrarily any total order $<$ on $S_N$, and define the semisimplicial object in affine $\CC$-schemes given by the diagram
\[ \Cech(\mathscr U) = \left( 
\begin{tikzcd}
\dots \rar[->,shift left=4pt]\rar[->,shift left=0pt]\rar[->,shift right = 4pt]&
\bigsqcup_{\substack{\sigma_1,\sigma_2 \in S_N\\ \sigma_1<\sigma_2}} U_{\prec_{\sigma_1}\wedge \prec_{\sigma_2}} \rar[->,shift left=2pt]\rar[->,shift right=2pt]&
\bigsqcup_{\sigma\in S_N } U_{\prec_\sigma}
\end{tikzcd}\right);\]
this diagram has a colimit in the category of $\CC$-schemes, and we define $\RavConf_N$ to be that colimit:
\begin{equation} \RavConf_N \coloneq \colim_{\Sch_\CC} \Cech(\mathscr U) .\nn
\end{equation}

Informally then, $\RavConf_2$ is a copy of $\AA_\CC^2$ but with a doubled diagonal; $\RavConf_3$ is a copy of $\AA_\CC^3$ but with all pairwise diagonals doubled and the main diagonal $z_1=z_2=z_3$ having multiplicity $3!=6$; and so on.\footnote{We emphasize that, like the formal raviolo $\Rav$ itself, this configuration space $\RavConf_N$ is not separated as a scheme for any $N\geq 2$; its underlying topological space is not Hausdorff.}

\subsection{Derived sections and the Thom-Sullivan functor}\label{sec: derived sections and the Thom-Sullivan functor}
Our goal is now to give a model in dg commutative algebras of the derived global sections $R\Gamma(\O,\RavConf_N)$ of the structure sheaf on $\RavConf_N$.

Given any finite Leray cover $\mathscr U = \{U_i\}_{i\in I}$ of a $\CC$-scheme $X$, the derived global sections $R\Gamma(X,\O)$ of its structure sheaf $\O$ is the dg commutative algebra defined, up to zig-zags of quasi-isomorphisms, as the homotopy limit in dg commutative algebras
\[ R\Gamma(X,\O) = \holim\Gamma(\Cech(\mathscr U),\O) .\] 
The \dfn{Thom-Sullivan functor} $\Th^\bul$ provides one way of computing any such homotopy limit, i.e. the homotopy limit of any diagram given by a semicosimplicial object in dg commutative algebras. We recall the definition of this functor in \cref{sec: review of thom-sullivan} and refer the reader to \cite{HinicSchecDeformationTheoryLie1994}, \cite[Appendix A]{KapraInfinitedimensionalDgLie2021} or e.g. \cite{AlfonYoungHigherCurrentAlgebras2023a} for further discussion.

The model of the homotopy limit which the Thom-Sullivan functor produces can be understood as consisting of polynomial differential forms on a single $(|I|-1)$-simplex, valued in $\O(\bigcap_{i\in I} U_i)$, together with polynomial differential forms on every face of that simplex. Each face is labelled by some subset $S\subset I$, and the form on that face is valued in $\O(\bigcap_{i\in S} U_i)$. (These intersections are taken in $X$.) These forms are required to satisfy the natural compatibility conditions under pullbacks. 

Whenever, as is true in our case, the maps $\O(\bigcap_{i\in S} U_i)\to \O(\bigcap_{i\in I} U_i)$, $S \subset I$, are all embeddings of commutative algebras, then these compatibility conditions mean that the forms on the faces of the $(|I|-1)$-simplex are actually determined by the form in the bulk of the simplex itself. The model of the homotopy limit is then a dg commutative algebra of polynomial differential forms on the $(|I|-1)$-simplex, valued in $\O(\bigcap_{i\in I} U_i)$, subject to boundary conditions:
\begin{align} \Th^\bul( \Gamma(\Cech(\mathscr U),\O) ) &= \bigl\{ \omega \in \O(\bigcap_{i\in I} U_i) \ox \CC[u_i,\dd u_i]_{i\in I}\big/ \langle \sum_{i\in I} u_i - 1, \sum_{i\in I} \dd u_i \rangle \nn\\
  &\qquad\hskip-2mm: \omega|_{\{u_i=0 \forall i\in I \setminus S\}} \in \O(\bigcap_{i\in S} U_i) \ox \CC[u_i,\dd u_i]_{i\in S}\big/ \langle \sum_{i\in S} u_i - 1, \sum_{i\in S} \dd u_i \rangle\nn\\&\qquad\qquad\qquad\text{ for all nonempty subsets } S\subset I \bigr\}.\nn\end{align} 
(Here $\omega|_{\{u_i=0 \forall i\in I \setminus S\}}$ denotes the pullback.)

Thus, in our case, we arrive at the following. Define the dg commutative algebra $\A_N$ by
\begin{align} \A_N &\coloneq \bigl\{ \omega \in \B_N \ox \CC[u_\sigma, \dd u_\sigma]_{\sigma \in S_N}\big/ \langle \sum_{\sigma\in S_N} u_\sigma - 1, \sum_{\sigma\in S_N} \dd u_\sigma \rangle \nn\\
    &\qquad: \omega|_{\{u_\sigma=0 \forall \sigma\in S_N \setminus S\}}\in  
    \O(U_{\bigwedge_{\sigma\in S}\prec_{\sigma} }) \ox \CC[u_i,\dd u_i]_{i\in S}\big/ \langle \sum_{i\in S} u_i - 1, \sum_{i\in S} \dd u_i \rangle \nn\\&\qquad\text{ for all nonempty subsets } S\subset S_N  \bigr\}.\nn\end{align}
\begin{thm}
  This $\A_N$ is a model, in dg commutative algebras, of the derived global sections of the structure sheaf on the ravioli configuration space $\RavConf_N$:
  \[ \A_N \simeq R\Gamma(\RavConf_N,\O) .\]\qed
\end{thm}
Because this algebra $\A_N$ will play a central role for us, it is worth noting the following equivalent description.
For any distinct $i,j\in [1,N]$, let $S_N^{ij}\subset S_N$ denote the set of total orders on $[1,N]$ in which $i$ precedes $j$:
 \begin{equation} S_N^{ij} \coloneq \{ \sigma \in S_N: i\prec_\sigma j \}.
  \label{eq: S_N^{ij}}\end{equation}
 %Obviously we have $S_N = S_N^{ij} \sqcup S_N^{ji}$. 
 \begin{lem}\label{lem: AN def}
  The definition of $\A_N$ above is equivalent to 
\begin{align} 
  \A_N &= \bigl\{ \omega \in \B_N \ox \CC[u_\sigma, \dd u_\sigma]_{\sigma \in S_N}\big/ \langle \sum_{\sigma\in S_N} u_\sigma - 1, \sum_{\sigma\in S_N} \dd u_\sigma \rangle \nn\\
  &\qquad:\text{for all distinct $i,j\in [1,N]$, the pullback $\omega|_{\{u_\sigma=0 \forall \sigma\in S_N^{ij}\}}$}
  \nn\\&\qquad\qquad\text{is regular in $z_i-z_j$ } \bigr\}.\nn
\end{align}
\end{lem}
\begin{proof} 
The idea is that by imposing the boundary conditions at these particular faces of the simplex, we are in fact imposing all the boundary conditions in the definition of $\A_N$, because all other faces on which the boundary conditions are non-empty arise as intersections of these faces. 

To see this in detail, it is enough to check that, for every nonempty subset $S\subset S_N$ of the set of total orders on $[1,N]$, if $\omega$ obeys the boundary conditions given in the lemma, then it obeys the boundary condition given in the definition of $\A_N$ at the face $\{u_\sigma=0 \forall \sigma\in S_N \setminus S\}$ corresponding to $S$.
To that end, pick any such $S$ and let $\prec_S$ denote the finest common coarsening,
\[\prec_S\,\, \coloneq \bigwedge_{\sigma\in S} \prec_\sigma,\]
or in other words, the partial order in which $i\prec_S j$ if and only if $i\prec_\sigma j$ for all $\sigma\in S$. The set $S$ is then the intersection (in $S_N$) of the subsets $S_N^{ij}$ as $i,j$ range over all pairs of distinct elements of $[1,N]$ such that $i\prec_S j$: 
\[ S= \bigcap_{i\prec_S j} S_N^{ij}.\]
It follows that our hyperplane $\{ u_\sigma = 0 \forall \sigma\in S_N \setminus S\}$ is the intersection (in $\AA_\CC^{N!}$) of the hyperplanes $\{ u_\sigma = 0 \forall \sigma\in S_\N \setminus S_N^{ij}\}$ as $i,j$ range over all pairs of distinct elements of $[1,N]$ that are comparable with respect to $\prec_S$:
\begin{equation} \{ u_\sigma = 0 \forall \sigma\in S_N \setminus S\} = \bigcap_{i\prec_S j} \{ u_\sigma = 0 \forall \sigma\in S_N\setminus S_N^{ij}\} .\nn\end{equation}
Now, of course, $S_N\setminus S_N^{ij} = S_N^{ji}$, and $z_i-z_j = -(z_j-z_i)$. So we see that by imposing the boundary conditions in the statement of the lemma, we are thereby imposing the condition that the pullback of $\omega$ to the face $\{ u_\sigma = 0 \forall \sigma\in S_N \setminus S\}$ is regular in $z_i-z_j$ for all $i\prec_S j$. This is the boundary condition on that face in the definition of $\A_N$, as required.
\end{proof}

%The boundary conditions are written in, in a sense, a minimal way here: for example we have not explicitly written the boundary condition that the pullback of $\omega$ to each vertex of the simplex should be regular in all $z_i-z_j$, but that condition follows from the fact that every vertex is an intersection of faces on which we \emph{are} placing conditions: for all $\tau \in S_N$, \[%\{u_\tau=1\} = \{u_\sigma = 0 \forall \sigma\neq \tau\} = \bigcap_{i=1}^{N-1}\{u_{\sigma}=0 \forall \sigma \in S_N^{\tau(i),\tau(i+1)}\}.\]

\section{Rational ravioli coinvariants}\label{sec: rational ravioli coinvariants}
Having defined the ravioli configuration space $\RavConf_N$ and a model $\A_N$ of the derived global sections of its structure sheaf, we now define the ravioli analogues of rational coinvariants/conformal blocks from \cref{sec: usual rational conformal blocks}. As far as possible, we shall follow the same construction of spaces of coinvariants with movable marked points we reviewed starting in \cref{sec: movable punctures in one complex dimension}, with $\RavConf_N$ playing the role of $\Conf_N$ and $\A_N$ playing the role of $\B_N$. 

%The upshot will be the definition a dg $\A_N$-module
%\begin{equation} \coinv(\g;\A_N;M_1,\dots,M_N)\nn\end{equation}
%of \dfn{ravioli coinvariants}, equipped with a quasi-isomorphism of $\A_N$-modules to  $\A_N \ox (M_1 \ox \dots \ox M_N)$, together with a notion of \dfn{taking coinvariants} in the raviolo setting: see \cref{ravioli-coinvariants} below and compare \cref{eq: taking coinvariants}. 

In what follows, the dg commutative algebra \[\A_N\simeq R\Gamma(\RavConf_N,\O)\] will play the role of the base ring (i.e. we shall work in dg $\A_N$-modules) in the same way that the commutative algebra \[\B_N = \Gamma(\Conf_N,\O)\] was the base ring (i.e. we worked in $\B_N$-modules) in the usual setting in \cref{sec: movable punctures in one complex dimension} of rational coinvariants with $N$ movable marked points.

%\begin{rem} We expect that the dg $\A_N$-module $\coinv(\g;\A_N;M_1,\dots,M_N)$ of ravioli coinviants will admit an interpretation as the derived global sections of a certain trivial dg vector bundle over $\RavConf_N$. Cf. \cref{rem: bundles} and \cref{rem: interpretation}. \end{rem}

Now, the commutative algebra $\B_{N+1} \cong \B_{N}[w,(w-z_i)^{-1}]_{i=1}^N$ was naturally a module over $\B_N$. %, i.e. it is a commutative algebra in $\B_N$-modules.
We want something similar in the ravioli setting. Namely we would like to show that $\A_{N+1}$ is a commutative algebra in dg $\A_N$-modules in some natural way.

To that end, let 
\begin{equation}\pi_N:S_{N+1} \to S_N;\quad \sigma \mapsto  (\sigma(1),\sigma(2),\dots,\widehat{N+1},\dots,\sigma(N+1))\label{def: piN}\end{equation} 
denote the surjective map of sets which sends total orders on $[1,N+1]$ to total orders on $[1,N]$ by forgetting about the position of $N+1$. We define a map of dg commutative algebras
\[ \iota_N: \CC[u_\sigma,\dd u_\sigma]_{\sigma\in S_{N}} \to \CC[u_\sigma,\dd u_\sigma]_{\sigma\in S_{N+1}}\]
by setting 
\begin{equation} \iota_N(u_\sigma) \coloneq  \sum_{\tau\in \pi_N^{-1}(\sigma)} u_\tau . \label{eq: iota_N}\end{equation}
For example, when $N=2$, we have
\[ \iota_2(u_{(12)}) = u_{(123)} + u_{(132)} + u_{(312)}, \qquad 
\iota_2(u_{(21)}) = u_{(213)} + u_{(231)} + u_{(321)}.\]

\begin{lem}\label{lem: AN action on AN+1} There is an injective map of dg commutative algebras 
\[\iota_N: \A_N \into \A_{N+1},\] 
(overloading notation somewhat) given by the tensor product of $\iota_N$ above with the obvious embedding $\B_N \into \B_{N+1}$.
\end{lem}
\begin{proof}
First observe that $\iota_N$ maps the dg ideal $\langle \sum_{\sigma\in S_N} u_\sigma - 1, \sum_{\sigma\in S_N} \dd u_\sigma \rangle$ to the dg ideal $\langle \sum_{\tau\in S_{N+1}} u_\tau - 1, \sum_{\tau\in S_{N+1}} \dd u_\tau \rangle$, so it defines a map between the polynomial differential forms on the $(N!-1)$-simplex and those on the $((N+1)!-1)$-simplex. Now we need to check that this map respects the defining boundary conditions of $\A_{N+1}$. It is enough to consider elements $\omega\in \A_N$ of the form
\[ \omega = p \ox \lambda, \quad p\in \B_{N}, \quad \lambda \in \CC[u_\sigma,\dd u_\sigma]_{\sigma\in S_{N}}.\]
%Suppose $p$ is singular in $z_i-z_j$ for some $i,j\in [1,N]$. Let $S\subset S_{N+1}$ index any collection of total orders on $[1,N+1]$ such that the elements $i$ and $j$ are comparable with respect the finest common coarsening, $\prec \coloneq \bigwedge_{\sigma \in S} \prec_{\sigma}$. We must check that $\iota_N(\lambda)$ vanishes on pullback to the face of the $(N+1)!-1$-simplex given by setting $u_{\sigma} =0$ for every $\sigma \in S_{N+1} \setminus S$. Without loss of generality, suppose that $i\prec j$ (if not, relabel $i\leftrightarrow j$).  It follows that $i\prec_{\sigma} j$ for every $\sigma \in S$. In other words, given any $\sigma \in S_{N+1}$, if $j \prec_{\sigma} i$, then $\sigma \in S_{N+1} \setminus S$. Therefore in particular, for every $\tau\in S_N$ on $[1,N]$ such that $j \prec_{\tau} i$, we have $\pi_N^{-1}(\tau) \subset S_{N+1} \setminus S$. It follows that the operation of pulling back to the zero locus of all the $u_\sigma$ with $\sigma \in S_{N+1} \setminus S$ factors through the operation of pulling back first to the zero locus of the images $\iota_N(u_\tau)$ of all those $u_{\tau}$ labelled by $\tau\in S_N$ such that $j \prec_{\tau} i$. But that latter pullback vanishes, since by assumption $\omega$ obeys the defining boundary conditions of $\A_N$.

Obviously, $p$ is not singular in $z_i-z_{N+1}$ for any $i$. Pick any distinct $i,j\in [1,N]$ and suppose that $p$ is singular in $z_i-z_j$. We must check that $\iota_N(\lambda)$ vanishes on pullback to the face of the $(N+1)!-1$-simplex given by $\{u_\tau = 0 \forall \tau \in S_{N+1}^{ij}\}$. On that zero locus, we have that $\iota_N(u_\sigma) = 0$ for every $\sigma \in S_N^{ij}$. That is, the operation of pulling back to the zero locus of all the $u_\tau$ with $\tau \in S_{N+1}^{ij}$ factors through the operation of pulling back first to the zero locus of the images $\iota_N(u_\sigma)$ with $\sigma \in S_N^{ij}$. That latter pullback commutes with the dg algebra map $\iota_N$, i.e. $\iota_N(\lambda)|_{\{\iota_N(u_\sigma) = 0 \forall \sigma \in S_N^{ij}\}} = \iota_N(\lambda|_{\{u_\sigma = 0 \forall \sigma \in S_N^{ij}\}})$. And finally, the pullback $\lambda|_{\{u_\sigma = 0 \forall \sigma \in S_N^{ij}\}}$ vanishes, since by assumption $\omega$ obeys the defining boundary conditions of $\A_N$.
\end{proof}

For example, when $N=2$, we have the following well-defined element of $\A_2$:
\[ \frac{u_{(12)} u_{(21)}}{z_1-z_2} \in \A_2, \]
and its image is a well-defined element of $\A_3$:
\[ \frac{(u_{(123)} + u_{(132)} + u_{(312)})(u_{(213)} + u_{(231)} + u_{(321)})}{z_1-z_2} \in \A_3 .\]

In this way, $\A_{N+1}$ has the structure of a commutative algebra not just in dg vector spaces but in dg $\A_N$-modules. 

In passing, let us note that the map $\iota_N: \A_N \into \A_{N+1}$ of \cref{lem: AN action on AN+1} has a natural family of generalizations. 
Given any subset $J\subset [1,N]$ we have a corresponding surjection 
\begin{equation} \pi_{J\subset [1,N]}: S_{N} \onto S_{|J|} \nn\end{equation}
between the sets of total orders (given by forgetting about the positions of elements of $[1,N]\setminus J$) and hence a map of dg commutative algebras
\begin{equation} 
\iota_{J\subset [1,N]}:\CC[u_\sigma,\dd u_\sigma]_{\sigma\in S_{|J|}} \to \CC[u_\sigma,\dd u_\sigma]_{\sigma\in S_{N}}\label{def: iota j}
\end{equation}
On tensoring this with the obvious embedding of commutative algebras
\[ \B_{|J|} \xrightarrow\cong \CC[z_i,(z_i-z_j)^{-1}]_{i,j\in J, i\neq j} \into \B_N \]   
we obtain an embedding of dg commutative algebras
\begin{equation} \iota_{J\subset [1,N]}:\A_{|J|} \into \A_N. \end{equation}
For example  we have the three maps 
\[ \iota_{\{1,2\}\subset [1,3]}, \iota_{\{1,3\}\subset [1,3]}, \iota_{\{2,3\}\subset [1,3]}: \A_2 \into \A_3. \]

\begin{exmp}
Such maps are a rich source of interesting elements of $\A_N$. Let $i,j\in [1,N]$ be distinct. Then the elements
\[ \frac{\dd \iota_{\{i,j\}}(u_{(12)})}{z_i-z_j} \]
belong to $\A_N$. Compare \cref{example: OS-like elements} below.
\end{exmp}

\subsection{Expansion maps}\label{sec: expansion-maps}
%Let \[\iota_{r \to s}: \B_N = \CC[z_i,(z_i-z_j)^{-1}]_{\substack{1\leq i,j\leq N\\ i\neq j}} \to \CC[z_i,(z_i-z_j)^{-1}]_{\substack{1\leq i,j\leq N\\ r\neq i\neq j\neq r}}((z_r-z_s))\] be the formal Laurent expansion map in small \(z_r-z_s\) with \(z_j, j\neq r\), held fixed. 
%Let \(f_{r\to s}:[1,N]\to [1,N]\setminus \{r\}\) denote the map given by \[f_{r\to s}:[1,N]\to [1,N]\setminus \{r\} ;\qquad f_{r\to s}(i) = \begin{cases} i & i \neq r \\ s & i = r\end{cases}.\]

Recall that we write  
\[\iota_{w\to z_s}: \B_{N+1} \to \B_N((w-z_s)) \] 
for the map of commutative algebras in $\B_N$-modules given by taking the Laurent expansion in small \(w-z_s= z_{N+1}-z_s\) with $z_1,\dots,z_N$ held fixed. (To keep track of the distinguished role of the last coordinate, we shall continue to write 
\(w \coloneq z_{N+1}\).)
One thinks of $\B_N((w-z_s))$ as a certain completion of the tensor product
\be \B_N \ox_\CC \CC((w-z_s)) \cong \Gamma(\Conf_N,\O)\ox_\CC \Gamma(D^\times,\O).\nn\ee
In a similar spirit, let us now introduce a commutative algebra in dg $\A_N$-modules $\A_N\sql w-z_s\sqr$ which we think of as a certain completion of the tensor product
\be \A_N \ox_\CC \CC\sql w-z_s\sqr \simeq R\Gamma(\RavConf_N,\O)\ox_\CC R\Gamma(\Rav,\O).\nn\ee
Namely, we define
\begin{align} \MoveEqLeft\A_N\sql w-z_s\sqr \coloneq \nn\\& \bigl\{ \omega \in \B_N((w-z_s)) \ox \CC[v,\dd v] \ox \CC[u_\sigma, \dd u_\sigma]_{\sigma \in S_N}\big/ \langle \sum_{\sigma\in S_N} u_\sigma - 1, \sum_{\sigma\in S_N} \dd u_\sigma \rangle \nn\\
  &\qquad:\text{for all distinct $i,j\in [1,N]$, the pullback $\omega|_{\{u_\sigma=0 \forall \sigma\in S_N^{ij}\}}$  }\nn\\&\qquad\qquad\text{is regular in $z_i-z_j$, } \nn\\&\qquad \qquad \text{ and both $\omega|_{v=0}$ and $\omega|_{v=1}$ are regular in $w-z_s$} \bigr\}.\nn\end{align}  

We think of $\A_N\sql w-z_s \sqr$ as the dg commutative algebra of functions on the formal raviolo near the $s$th of $N$ distinguishable movable marked points, in the same way that $\B_N((w-z_s))$ is the commutative algebra of functions on the punctured formal disc near the $s$th of $N$ distinguishable movable marked points, cf. \cite[\S13.2]{FrenkelBenZvi}.  

Now we shall define a map of commutative algebras in dg $\A_N$-modules
\[ \A_{N+1} \to \A_N\sql w-z_s\sqr\]
given by Laurent expanding in $w-z_s$ and simultaneously pulling back to a certain judiciously chosen curved copy of $\Delta_\CC^1\times \Delta_\CC^{N!-1}$ inside the algebro-geometric $((N+1)!-1)$-simplex
\[ \Delta_\CC^{(N+1)!-1} \coloneq \Spec\left( \CC[u_\sigma]_{\sigma\in S_{N+1}}\big/ \langle \sum_{\sigma\in S_{N+1}} u_\sigma - 1 \rangle\right)  \into \AA_\CC^{(N+1)!}.\]
Namely, we first define a map of affine schemes
\[ \AA_\CC^1 \times \AA_\CC^{N!} = \Spec\left(\CC[v] \ox \CC[u_\sigma]_{\sigma\in S_{N}}\right) \to \AA_\CC^{(N+1)!} = \Spec \CC[u_\sigma]_{\sigma\in S_{N+1}} \]
or equivalently a map of commutative algebras
\[ \CC[u_\sigma]_{\sigma\in S_{N+1}} \to \CC[v] \ox \CC[u_\sigma]_{\sigma\in S_{N}}\]
by sending
\begin{equation}
  u_\sigma \mapsto \begin{cases} 
    (1-v) u_{(\dots,s,\dots)} & \text{if } \sigma = (\dots ,N+1,s,\dots) \\
    v u_{(\dots,s,\dots)} & \text{if } \sigma = (\dots ,s,N+1,\dots) \\
    0 & \text{otherwise.} \end{cases} \label{eq: pullback to product of simplices} \end{equation}
Since 
\[ \sum_{\sigma\in S_{N+1}} u_\sigma \mapsto \sum_{\sigma\in S_{N}} \left( v u_\sigma + (1-v) u_\sigma  \right)  = \sum_{\sigma\in S_{N}} u_\sigma, \]
this map induces a map of affine schemes
\[ p_{N+1\to s}: \Delta_\CC^1 \times \Delta_\CC^{N!-1} \to \Delta_\CC^{(N+1)!-1}.\]
One should think of this as embedding a curved copy of $\Delta_\CC^1\times \Delta_\CC^{N!-1}$ into $\Delta_\CC^{(N+1)!-1}$. We get also the map $p_{N+1\to s}^*$ of dg commutative algebras from the polynomial differential forms on $\Delta_\CC^{(N+1)!-1}$ to the polynomial differential forms on $\Delta_\CC^1 \times \Delta_\CC^{N!-1}$,
\begin{align} p_{N+1\to s}^* &: \CC[u_\sigma,\dd u_\sigma]_{\sigma\in S_{N+1}}\big/ \langle \sum_{\sigma\in S_{N+1}} u_\sigma - 1, \sum_{\sigma\in S_{N+1}} \dd u_\sigma \rangle\nn\\&\qquad\qquad
\to \CC[v,\dd v] \ox \CC[u_\sigma,\dd u_\sigma]_{\sigma\in S_{N}}\big/\langle \sum_{\sigma\in S_{N}} u_\sigma - 1, \sum_{\sigma\in S_{N}} \dd u_\sigma \rangle, \nn\end{align}
given again by \cref{eq: pullback to product of simplices}.

We can now define an analogue of the Laurent-expansion map $\iota_{w\to z_s}$ for $\A_N$. (Here we shall overload the notation $\iota_{w\to z_s}$ somewhat.) 
\begin{defprop}\label{defprop: iota}
There is a map of commutative algebras in dg $\A_N$-modules  
\[ \iota_{w\to z_s} : \A_{N+1} \to \A_N\sql w-z_s\sqr\]
given by $\iota_{w\to z_s} \ox p_{N+1\to s}^*$. That is, we take the formal Laurent expansion in small \(w-z_s\) with $z_1,\dots,z_N$ held fixed, and we pull back to the copy of $\Delta_\CC^1 \times \Delta_\CC^{N!-1}$ in $\Delta_\CC^{(N+1)!-1}$ defined by the map $p_{N+1\to s}$ above. 
\end{defprop}
\begin{proof}
We certainly have a map of commutative algebras in dg vector spaces
\[ \A_{N+1} \to \B_N((w-z_s)) \ox \CC[v,\dd v] \ox \CC[u_\sigma, \dd u_\sigma]_{\sigma \in S_N}\big/ \langle \sum_{\sigma\in S_N} u_\sigma - 1, \sum_{\sigma\in S_N} \dd u_\sigma \rangle .\]
What has to be checked is first that this map respects the defining boundary conditions of $\A_{N}\sql w-z_s\sqr$ and second that it is a map of $\A_{N}$-modules. 

Let $\omega \in \A_{N+1}$ be any element. We must show that $\iota_{w\to z_s}(\omega)$ obeys the defining boundary conditions of $\A_{N}\sql w-z_s\sqr$. 

Consider first $(\iota_{w\to z_s}(\omega))|_{v=0}$. On the  preimage of the zero locus of $v$ under the map $p^*_{N+1\to s}$, we have that $u_\tau$ vanishes for all $\tau \in S_{N+1}$ except for those of the form $(\dots,N+1,s,\dots)$. Thus, in particular, on that preimage we have that $u_\tau = 0$ for all $\tau \in S_{N+1}^{s,N+1}$. Therefore the pullback of $\omega$ to that preimage is regular in $w-z_s$ by virtue of the defining boundary conditions of $\A_{N+1}$, and hence  $(\iota_{w\to z_s}(\omega))|_{v=0}$ is regular in $w-z_s$, as required.

The argument for $(\iota_{w\to z_s}(\omega))|_{v=1}$ is similar. 

Next consider $(\iota_{w\to z_s}(\omega))|_{u_{\sigma}=0 \forall \sigma\in S_{N}^{ij}}$ for distinct $i,j \in [1,N]\setminus\{s\}$. We must show that this is regular in $z_i-z_j$. When we set to zero $u_\sigma$ for all $\sigma \in S_N^{ij}$, we are thereby setting to zero the images $p_{N+1\to s}^*(u_\tau)$ of $u_\tau$ for all $\tau \in S_{N+1}^{ij}$. Therefore the pullback $\omega|_{(p_{N+1\to s}^*)^{-1}(\{u_\sigma=0\forall \sigma \in S_N^{ij}\})}$ of $\omega$ to this preimage is regular in $z_i-z_j$, by virtue of the defining boundary conditions of $\A_{N+1}$. Hence $(\iota_{w\to z_s}(\omega))|_{u_{\sigma}=0 \forall \sigma\in S_{N}^{ij}}$ is regular in $z_i-z_j$, again as required.

Finally we consider $(\iota_{w\to z_s}(\omega))|_{u_{\sigma}=0 \forall \sigma\in S_{N}^{is}}$ for $i \in [1,N]\setminus\{s\}$. When we set to zero $u_\sigma$ for all $\sigma \in S_N^{is}$, we are thereby setting to zero the images $p_{N+1\to s}^*(u_\tau)$ of $u_\tau$ \emph{both} for all $\tau \in S_{N+1}^{is}$ \emph{and} crucially also for all $\tau \in S_{N+1}^{i,N+1}$. Therefore the defining boundary conditions of $\A_{N+1}$ guarantee that the pullback $\omega|_{(p_{N+1\to s}^*)^{-1}(\{u_\sigma=0\forall \sigma \in S_N^{is}\})}$ of $\omega$ to this preimage is regular in both $z_i-z_s$ and crucially also in $z_i-w$. Hence $(\iota_{w\to z_s}(\omega))|_{u_{\sigma}=0 \forall \sigma\in S_{N}^{is}}$ is regular in $z_i-z_s$, again as required. (One should keep in mind that the process of taking Laurent expansions introduces additional singularities. For example
\[\iota_{w\to z_s} \frac{1}{w-z_i} = \sum_{k=0}^\8 (-1)^k\frac{1}{(z_s-z_i)^{k+1}} (w-z_s)^k .\]
These are dealt with by the ``and crucially'' part of the argument above.)

The argument for $(\iota_{w\to z_s}(\omega))|_{u_{\sigma}=0 \forall \sigma\in S_{N}^{si}}$ is similar. 

%Concerning the boundary conditions, the key observation is this: for any distinct $i,j\in[1,N] \setminus \{s\}$, when we set to zero $u_\sigma$ for all $\sigma \in S_N^{ij}$, we are thereby setting to zero the images $p_{N+1\to s}^*(u_\tau)$ of $u_\tau$ for all $\tau \in S_{N+1}^{ij}$. Moreover, for any $i\in [1,N]\setminus\{s\}$, when we set to zero $u_\sigma$ for all $\sigma \in S_N^{is}$, we are thereby setting to zero the images $p_{N+1\to s}^*(u_\tau)$ of $u_\tau$ \emph{both} for all $\tau \in S_{N+1}^{is}$ \emph{and} crucially also for all $\tau \in S_{N+1}^{i,N+1}$. Similarly for $S_N^{si}$. (By ``setting to zero'', we mean more precisely ``pulling back to the zero locus of''.) In this way, we see that if $\omega$ obeys the defining boundary conditions of $\A_{N+1}$, then $\iota_{w\to z_s}(\omega)$ obeys the defining boundary conditions of $\A_{N}\sql w-z_s\sqr$. (One should keep in mind that the process of taking Laurent expansions introduces additional singularities. For example\[\iota_{w\to z_s} \frac{1}{w-z_i} = \sum_{k=0}^\8 (-1)^k\frac{1}{(z_s-z_i)^{k+1}} (w-z_s)^k .\]
%These are dealt with by the ``and crucially'' part of the argument above.)

It remains to check that the map is a map of dg $\A_N$-modules, cf. \cref{lem: AN action on AN+1}. But this follows from the observation that, for every $\sigma \in S_N$, we have
\[ p_{N+1\to s}^* (\iota_N(u_\sigma)) = \sum_{\tau\in \pi_N^{-1}(\sigma)} p_{N+1\to s}^*(u_\tau) = v u_\sigma + (1-v) u_\sigma = u_\sigma .\qedhere\]
\end{proof}

\subsection{Cospan of dg Lie algebras}\label{sec: lie-algebra-splitting}
Let us define $\A_{N+1}'$ to be the nonunital dg subalgebra of $\A_{N+1}$ consisting of those elements vanishing as $w\to \8$.
At this stage we have maps of commutative algebras in dg $\A_N$-modules
\[ \A_{N+1}' \xrightarrow{} \bigoplus_{k=1}^N \A_N\sql w-z_k\sqr \xleftarrow{} \bigoplus_{k=1}^N \A_N\sql w-z_k\sqr_+ .\]
(Here $\A_N\sql w-z_k\sqr_+:= \A_N[[w-z_k]] \ox \CC[v,\dd v]$, as in \cref{sec: splitting}.)

Let $\g$ be a simple finite-dimensional Lie algebra over $\CC$, as earlier. 
On tensoring with $\g$ we get maps of Lie algebras in dg $\A_N$-modules. Namely, let us define
\[ \a \coloneq \g \ox \bigoplus_{k=1}^N \A_N\sql w-z_k\sqr, \qquad \a_+ \coloneq \g \ox \bigoplus_{k=1}^N \A_N\sql w-z_k\sqr_+, \qquad \a_- \coloneq \g \ox \A_{N+1}' .\]
Then we have the cospan of Lie algebras in dg $\A_N$-modules
\[ \a_- \xrightarrow{\iglob} \a \xleftarrow{\iloc} \a_+. \]

This is analogous to the cospan of Lie algebras in $\B_N$-modules we had in \cref{sec: usual rational conformal blocks} above. (In what follows we shall often omit the map $\iloc$ and simply identify elements of $\a_+$ with their embedded images in $\a$.) 

Moreover, we still have the following.
\begin{prop}\label{prop: a is a sum}
   As a dg $\A_N$-module, $\a$ is the sum (although not, as we shall see, the direct sum) of the images of $\a_+$ and $\a_-$:
\[ \a = \a_+ + \iglob(\a_-) .\]
\end{prop}
\begin{proof}
We must show that every element of $\a$ can be written as a sum of an element of $\a_+$ and an element of $\iglob(\a_-)$. Let us define 
\[ \a^\pm := \g \ox \bigoplus_{i=1}^N \A_N\sql w-z_i\sqr_\pm \] 
where on the right $\pm$ means restricting to non-negative (respectively, strictly negative) powers of $w-z_i$. Thus, $\a^+ \equiv \a_+$, but of course $\a^- \neq \iglob(\a_-)$. At the level of dg $\A_N$-modules we evidently have
\( \a = \a^- \oplus \a^+. \) 
Given any element $X \in \a$, let $X = X^+ + X^-$ be its corresponding decomposition. It is enough to show that $X^-$ is in the image of $\iglob(\a_-)$ modulo terms in $\a^+$.
In other words, it is enough to construct a map of dg $\A_N$-modules 
\[ g: \a^- \to \a_-, \] 
the ``building global objects'' map, with the property that $(\iglob(g(X^-)))^-=X^-$ for all $X \in \a$. 
To that end, we first note that the map 
\[ p_{N+1\to s} : \Delta_\CC^1 \times \Delta_\CC^{N!-1} \to \Delta_\CC^{(N+1)!-1} \]
we defined above has a left inverse
\begin{equation} q_{s} : \Delta_\CC^{(N+1)!-1} \to \Delta_\CC^1 \times \Delta_\CC^{N!-1} \label{qdef}\end{equation}
given by (here, recall \cref{def: iota j})
\[ q_{s}^* := \iota_{\{s,N+1\}\subset [1,N+1]} \ox \iota_{[1,N]\subset [1,N+1]} .\]
That is, explicitly, we have
\[ q_{s}^*(v) = \sum_{\tau \in S_{N+1}^{s,N+1}} u_\tau ,\qquad q_{s}^*(u_\sigma) = \iota_N(u_\sigma),\,\,\sigma \in S_N, \]
with $\iota_N$ as in \cref{eq: iota_N} and $S_{N+1}^{s,N+1}$ as in \cref{eq: S_N^{ij}}.
To see that $q_{s}$ is left inverse to $p_{N+1\to s}$ we note first that $p_{N+1\to s}^* q_{s}^* (u_\sigma) = p_{N+1\to s}^* (\iota_N(u_\sigma)) = u_\sigma$ for every $\sigma \in S_N$, as we checked in the proof of \cref{defprop: iota} above; and second that 
\[ p_{N+1\to s}^* q_{s}^* (v) 
= p_{N+1\to s}^* \left( \sum_{\tau \in S_{N+1}^{s,N+1}} u_\tau \right) 
 = v\sum_{\sigma \in S_N} u_\sigma =v.\]
where in the last equality we have used the defining relation $\sum_{\sigma \in S_N} u_\sigma = 1$ of the algebro-geometric simplex $\Delta_\CC^{N!-1}$.

Now, suppose we are given an element $\omega \in \A_N\sql w-z_k\sqr_-$. Such an $\omega$ is in particular a polynomial differential form valued in $(w-z_k)^{-1} \B_N[(w-z_k)^{-1}]$. We apply to it the map of commutative algebras in $\B_N$-modules
\be (w-z_k)^{-1} \B_N[(w-z_k)^{-1}] \into \B_{N+1}'.\nn\ee 
The result is an element of the dg commutative algebra
\[ \B_{N+1}' \ox \CC[v,\dd v] \ox \CC[u_\sigma,\dd u_\sigma]_{\sigma\in S_N}\big/ \langle \sum_{\sigma\in S_N} u_\sigma - 1, \sum_{\sigma\in S_N} \dd u_\sigma \rangle \]
obeying certain boundary conditions. We may map it to the dg commutative algebra 
\[ \B_{N+1}' \ox \CC[u_\sigma,\dd u_\sigma]_{\sigma\in S_{N+1}}\big/ \langle \sum_{\sigma\in S_{N+1}} u_\sigma - 1, \sum_{\sigma\in S_{N+1}} \dd u_\sigma \rangle \] 
via the map $q_{s,N}^*$. Let us check that the resulting form, call it $\tilde\omega$, obeys the defining boundary conditions of $\A_{N+1}$. First, consider singularities in $z_i-z_j$ for any $i,j\in [1,N]$ with $i\neq j$. We must consider the pullback of $\tilde\omega$ to $\{u_\sigma=0 \forall \sigma \in S_{N+1}^{ij}\}$. On the latter zero locus, we have $q_{s,N}^*(u_\tau)=0$ for every $\tau \in S_N^{ij}$. Therefore this pullback of $\tilde\omega$ is regular in $z_i-z_j$, since $\omega$ obeyed the boundary conditions of $\A_N$. Next, consider singularities in $w-z_s$. We are to consider the pullback of $\tilde\omega$ to $\{u_\sigma=0 \forall \sigma \in S_{N+1}^{s,N+1}\}$ (or the same with $S_{N+1}^{N+1,s}$, for which the argument is similar). On that zero locus, $q_{s}^*(v)$ vanishes. Therefore $\tilde\omega$ is regular in $w-z_s$ there, since $\omega$ obeyed the boundary conditions of $\A_N\sql w-z_s\sqr$. Finally note that $\tilde\omega$ is obviously regular everywhere in $z_i-w$ for $i\neq s$. 

In this way we obtain a map of dg $\A_{N}$-modules
\begin{equation} g_k: \A_N\sql w-z_k\sqr_- \to \A_{N+1}' \label{def: gk}\end{equation}
for each $k$. For future use, let us remark that for each individual $k$, the map $g_k$ is even a map of commutative algebras in dg $\A_{N}$-modules. (It is the analogue in our raviolo context of the map $(w-z_k)^{-1}\B_N[(w-z_k)^{-1}] \to \B_{N+1}'$ above.)

Hence we obtain a map of dg $\A_N$-modules
\[ g: \g \ox \bigoplus_{k=1}^N \A_N\sql w-z_k\sqr_- \to \g \ox \A_{N+1}' .\]
By construction, we have that $(\iglob(g(X)))^-=X^-$ for all $X \in \a$. Thus, finally, 
\be X = \iglob(g(X)) + (X - \iglob(g(X))) \in \iglob(\a_-) + \a^+ \ee
for every $X\in \a$, as required. 
\end{proof}

The map $\iglob$ has a nontrivial kernel.\footnote{For example, consider the case $N=2$, $N+1=3$. The element 
  \[ \frac{u_{(123)} u_{(312)}}{w-z_2} \]
  belongs to $\A_3'$, and it is in the kernel of the expansion map $\A_3 \to \A_2\sql w-z_1\sqr\oplus \A_2\sql w-z_2\sqr$. Indeed, $p_{3\to 1}^*(u_{(123)}) = 0$ and $p_{3\to 2}^*(u_{(312)}) = 0$.
  
  (Observe, in passing, that for example neither $\frac{u_{(123)}}{w-z_2}$ nor $\frac{\dd u_{(123)}}{w-z_2}$ belong to $\A_3$: they are not regular in $w-z_2$ on pullback to $u_{(321)}=u_{(312)} = u_{(132)} =0$.) 
}
Of course, one could always simply \emph{define} the subalgebra $\A_{N+1}'$ to be the quotient of $\A_{N+1}$ by the kernel of $\iglob$ (a dg ideal). However, the following example illustrates a more profound disanalogy between the usual case and the ravioli case.

\begin{exmp}\label{example: OS-like elements}
Consider the case $N=2$, $N+1=3$. 
Let us write
\[ v_{12} \coloneq \iota_{\{1,2\}}(u_{(12)}) = u_{(123)} + u_{(132)} + u_{(312)}, \]
\[ v_{13} \coloneq \iota_{\{1,3\}}(u_{(12)}) = u_{(123)} + u_{(213)} + u_{(132)}, \]
\[ v_{23} \coloneq \iota_{\{2,3\}}(u_{(12)}) = u_{(123)} + u_{(213)} + u_{(231)} \] 
(and $v_{32} = 1- v_{23}$ etc.). 
  Consider the element 
  \[ \Omega_{12}:= \frac{ \dd v_{31}}{w-z_1} \wedge \frac{\dd v_{32}}{w-z_2} -
  \frac{ \dd v_{21}}{z_2-z_1} \wedge \frac{\dd v_{32}}{w-z_2} -
  \frac{ \dd v_{31}}{w-z_1} \wedge \frac{\dd v_{12}}{z_1-z_2} \in \A_3' .\]
It is nonzero. The singular part of its expansion in small $w-z_1$ is 
\[\dd v_{31} \wedge \left(\dd v_{32} - \dd v_{12} \right)  \frac{ 1}{w-z_1}\frac 1 {z_1-z_2} . 
\]
Consider the pullback of this to the copy of $\Delta^1 \times \Delta^1$ in $\Delta^3$ defined by the map $p_{3;1}$ above. We have, on this copy of $\Delta^1 \times \Delta^1$,
\begin{align*}
p_{3\to1}^*(v_{12}) &= 0 + (1-v) u_{(12)} +v u_{(12)} = u_{(12)} \\
p_{3\to1}^*(v_{13}) &= 0 + v u_{(21)} + v u_{(12)} = v (u_{(21)} + u_{(12)}) = v \\
p_{3\to1}^*(v_{23}) &= 0 + v u_{(21)} + (1-v) u_{(21)} =  u_{(21)}.
\end{align*}
Therefore this pullback vanishes:
\[ - \dd v \wedge \dd ( u_{(12)} - u_{(12)} ) \frac{ 1}{w-z_1}\frac 1 {z_1-z_2}  =0 .\]
A similar story holds for the expansion in small $w-z_2$. 

%On the other hand, $\Omega_{12}$ is not in the kernel of the expansion map.

We conclude that the image of this element $\Omega_{12} \in \A_3'$ in $\A_2\sql w-z_1\sqr\oplus \A_2\sql w-z_2\sqr$, while nonzero, is nonsingular in the local variables $w-z_i$ in each summand, i.e. it lies in the subalgebra
$\A_2\sql w-z_1\sqr_+\oplus \A_2\sql w-z_2\sqr_+$. 
\end{exmp}

\begin{rem}
%This example shows several things. 
%First, the images of $\a_-$ and $\a_+$ in $\a$ have non-trivial intersection. 
The element $\Omega_{12}$ is closed and we expect that it represents a nontrivial cohomology class in $\A_3$. However we expect that its image in $\A_2\sql w-z_1\sqr\oplus \A_2\sql w-z_2\sqr$ is exact. That is, we expect that $H(\Omega_{12})$ is in the kernel of the induced map of cohomologies. 
%In some sense, the fact that $\Omega_{12}$ is not in the kernel on the nose is an inconvenience for us: if it were, we could just set such elements to zero by working with the quotient of $\A_3$ by the kernel of the embedding map $(\iota_{w\to z_1},\iota_{w\to z_2})$; morally, speaking, one could then ``take partial fractions'', as one can in the usual case. (Taking partial fractions in $w$ is one way to show that $\B_{N+1}$ is free over $\B_N$, for instance.)
\end{rem}
\begin{rem}
It is interesting to note the superficial similarity of the elements of which $\Omega_{12}$ is an example, namely  
\[ \frac{ \dd v_{N+1,j}}{w-z_j} \wedge \frac{\dd v_{jk}}{z_j-z_k}  + \frac{ \dd v_{jk}}{z_j-z_k} \wedge \frac{\dd v_{k,N+1}}{z_k-w}  + \frac{ \dd v_{k,N+1}}{z_k-w} \wedge \frac{\dd v_{N+1,j}}{w-z_j}, \]
with the relations 
  \[ \frac{ \dd z_i - \dd z_j}{z_i-z_j} \wedge \frac{\dd z_j - \dd z_k}{z_j-z_k}  + \frac{ \dd z_j - \dd z_k}{z_j-z_k} \wedge \frac{\dd z_k - \dd z_i}{z_k-z_i}  + \frac{ \dd z_k - \dd z_i}{z_k-z_i} \wedge \frac{\dd z_i - \dd z_j}{z_i-z_j}  = 0 \]
  which hold in the de Rham complex of holomorphic forms on the usual configuration space $\Conf_N = \AA_\CC^N \setminus \{\textrm{diagonals}\}= \Spec \B_N$, and which are examples of \emph{Orlik-Solomon} relations associated to a hyperplane arrangement. 
  See \cite{SchecVarchArrangementsHyperplanesLie1991} and e.g. \cite[\S2]{VarchYoungCyclotomicDiscriminantalArrangements2019}. 
\end{rem}
  
As \cref{example: OS-like elements} shows, the image of $\a_-$ in $\a$ has nontrivial intersection with $\a_+$. Let us define $\a_0$ to be this intersection, a Lie subalgebra in dg $\A_N$-modules of $\a$:
\[ \a_0 := \iglob(\a_-) \cap \a_+ \subset \a.\]

\begin{rem} 
Recall that in the usual case reviewed in \cref{sec: usual rational conformal blocks}, we were able to arrange that $\b_- \cap \b_+ = 0$ in $\b$, and hence that $\b = \b_- \oplus \b_+$ was the direct sum of vector spaces, or of $\B_N$-modules. We did that by defining $\b_- = \g \ox \B_{N+1}'$, where $\B_{N+1}'$ was the nonunital subalgebra of $\B_{N+1}$ consisting of those rational expressions in $w$ vanishing as $w\to \8$; that ensured, in \cref{movablelrembeddings}, that the image of $\B_{N+1}'$ had trivial intersection with the image of $\bigoplus_{i=1}^N \B_N[[w-z_i]]$. 
In the construction of rational conformal blocks, \cite{FeigiFrenkResheGaudinModelBethe1994}, it is common to define $\B_{N+1}'$ that way, for precisely this reason. Doing so yields what is called the \emph{modified space of conformal blocks} in \cite[\S13]{FrenkelBenZvi}. 
If one weakens that restriction on $\b_-$ then one also has a non-trivial intersection $\b_- \cap \b_+ =: \b_0$ in the usual setting, and the space of coinvariants has a residual quotient, 
\[ \coinv(\g;\B_N;M_1,\dots,M_N) \cong (\B_N \ox M)\big/\b_0 .\] 
For example, if one insists only that elements of $\b_-$ be \emph{regular} as $w\to \8$, then one finds $\b_0 = \B_N \ox \g$, the Lie algebra of zero modes. In that case, the space of rational coinvariants is  isomorphic as a $\B_N$-module to the quotient by the diagonal action of $\g$ on $M= \bigotimes_{i=1}^NM_i$:
\[ \coinv(\g;\B_N;M_1,\dots,M_N) \cong \B_N \ox \bigl(M\big/\g\bigr) .\] 
\end{rem}

In our present setting, there seems to be no obvious way to avoid this non-trivial intersection $\a_0$.
By the PBW theorem (which holds since $\A_N \supset \QQ$) we get that
\begin{equation} U_{\A_N}(\a) \cong U_{\A_N}(\iglob(\a_-)) \ox_{U_{\A_N}(\a_0)} U_{\A_N}(\a_+) \label{PBW}\end{equation}
as $(U_{\A_N}(\iglob(\a_-)),U_{\A_N}(\a_+))$-bimodules in dg $\A_N$-modules. 

That has the following somewhat awkward consequence: $U_{\A_N}(\a)$ is not free as a $(U_{\A_N}(\a_-),U_{\A_N}(\a_+))$-bimodule in dg $\A_N$-modules, and therefore we have no reason to expect $U_{\A_N}(\a)$ to be cofibrant in that category with respect to its projective model structure, cf. \cref{rem: derived coinvariants} below.

\subsection{Ravioli coinvariants}\label{sec: induced modules and coinvariants}
Let now $M_i$ be any smooth module over the dg Lie algebra $\g \ox \CC\sql w-z_i\sqr_+$, for each $1\leq i\leq N$. (By \dfn{smooth} we mean, following the usual case in \cite[\S5.1.5]{FrenkelBenZvi}, that for all $m\in M_i$ the Lie ideal $\g \ox (w-z_i)^k\CC\sql w-z_i\sqr_+$ acts as zero on $m$, for all sufficiently large $k$.)

Then $\A_N \ox M_i$ is a smooth module over  $\g \ox \A_N\sql w-z_i\sqr_+$, for each $i$, and  
%\begin{exmp} The following are examples of such modules. First let \(M_i\) be a $\g$-module, for each $1\leq i\leq N$. As usual, we make $M_i$ into a module over $\g \ox \CC[[w-z_i]]$ by declaring that $\bigl(\g \ox (w-z_i)\CC[[w-z_i]] \bigr)\on M_i = 0$, i.e. that strictly positive modes act as zero. We then further make $M_i$ into a module over $\g \ox \CC\sql w-z_i\sqr_+$ by declaring that $\g \ox \CC[[w-z_i]] \ox I$ acts as zero where $I$ is the dg ideal in $\CC[v,\dd v]$ generated by $v$. That is, the action of an element of $\g\ox \sql w-z_i\sqr_+$ on $M_i$ factors through taking the zero mode of the pullback of that element to $v=0$. Since this module $M_i$ is smooth over $\g \ox \CC\sql w-z_i\sqr_+$, we can extend scalars in an obvious way to obtain a module $\A_N\ox_\CC M_i$ over $\g \ox \A_N\sql w-z_i\sqr_+$. \end{exmp}
\[ \A_N \ox M \coloneq \A_N \ox \bigotimes_{i=1}^N M_i ,\]
is a left $U_{\A_N}(\a_+)$-module.
We have the induced module
\[ \mathscr M := U_{\A_N}(\a) \ox_{U_{\A_N}(\a_+)} (\A_N\ox M), \]
a left $U_{\A_N}(\a)$-module. It is equivalently, the tensor product (over $\A_N$) of the induced modules
\[ \mathscr M_i := U_{\A_N}(\g \ox \A_N\sql w-z_i\sqr) \ox_{U_{\A_N}(\g \ox \A_N\sql w-z_i\sqr_+)} (\A_N \ox M_i) \]
at the marked points. 

We may then define the space (more precisely, the dg $\A_N$-module) of \dfn{ravioli coinvariants} of $\g$ acting on $M_1,\dots,M_N$ to be
\[ \coinv(\g;\A_N;M_1,\dots,M_N) \coloneq \A_N \ox_{U_{\A_N}(\a_-)} U_{\A_N}(\a) \ox_{U_{\A_N}(\a_+)}  (\A_N\ox M) .\]
By the PBW theorem, \cref{PBW}, we have that
\begin{equation} \coinv(\g;\A_N;M_1,\dots,M_N) \cong \A_N \ox_{U_{\A_N}(\a_0)}  (\A_N\ox M) =:  (\A_N\ox M)/\a_0\label{def: coinv}\end{equation}
as left $\A_N$-modules. 

\begin{rem}\label{rem: derived coinvariants}
This definition has the merit of being relatively concrete (although we don't have an explicit description of $\a_0$). In principle however, one should really consider the derived tensor product
\begin{align} %R\coinv(\g;\A_N;M_1,\dots,M_N) \coloneq 
\A_N \lox_{U_{\A_N}(\a_-)} U_{\A_N}(\a) \lox_{U_{\A_N}(\a_+)} M ,\end{align}
which will be modelled by the tensor product
\[ \widetilde \coinv := \A_N \ox_{U_{\A_N}(\a_-)} QU_{\A_N}(\a) \ox_{U_{\A_N}(\a_+)} M \]
where $QU_{\A}(\a)$ is a cofibrant replacement of $U_{\A}(\a)$ in the category of $(U_{\A_N}(\a_-),U_{\A_N}(\a_+))$-bimodules in dg $\A_N$-modules equipped with its projective model structure \cite{HinichHomotopyAlgebras}.\footnote{That is, a $(U_{\A_N}(\a_-),U_{\A_N}(\a_+))$-bimodule in dg $\A_N$-modules is the same thing as a left dg $U_{\A_N}(\a_-)\ox_{\A_N}U_{\A_N}(\a_+^\mathrm{op}))$-module, so this is a special case of the model category structure on the category of (left) dg-modules over a (not necessarily graded-commutative) dg-algebra as defined in \cite{HinichHomotopyAlgebras}.} 
We expect \cref{thm: raviolo Y map} below to apply to (a suitable choice of model $\widetilde \coinv$ of) this derived space of coinvariants as well, but we do not show that here. %. In outline, the argument is the following. First, one shows that the map $U_{\A_N}(\a_-) \ox U_{\A_N}(\a_+) \to U_{\A_N}(\a)$ given by $a\ox b \mapsto \iglob(a) b$ is surjective in cohomology. 

\end{rem}

\section{Main result}\label{sec: main result}
We can now state the main result of the present paper, which says that the state-field map for the raviolo vacuum module $\mathscr V$, as we defined it in \cref{sec: raviolo state-field map}, emerges naturally when one considers appropriate limits of the spaces of coinvariants introduced in \cref{sec: induced modules and coinvariants} above. That is, \cref{prop: Y map} above continues to hold in the raviolo case, mutatis mutandis.

Indeed, let us again specialize to the case in which we insert a copy of the vacuum module at the $N-1$st and $N$th marked points, $M_{N-1} = \CC\vac$ and $M_N = \CC\vac$. 
By identifying local coordinates $w-z_i$ in the complex-algebraic direction, we may identify each of the dg Lie algebras $\g\ox \CC\sql w-z_i\sqr$ with a single copy $\g\ox \CC\sql s\sqr$, and thereby identify their vacuum Verma modules with a single abstract copy of $\mathscr V$:
\begin{equation} \mathscr V\cong \mathscr V_i \coloneq \Ind_{\g\ox \CC[[w-z_i]]}^{\g \ox \CC\sql w-z_i\sqr} \CC\vac. \nn\end{equation}
There is then an evident embedding map of dg vector spaces
\be \mathscr V_i \into %\prescript{}{\A}{\mathscr V}_i := 
U_{\A_N}(\g \ox \A_N\sql w-z_i\sqr) \ox_{U_{\A_N}(\g \ox \A_N\sql w-z_i\sqr_+)} \left(\A_N\vac\right) \nn\ee
coming from the unit map $1: \CC \into \A_N$. 

At the other sites, we pick arbitrary smooth modules $M_i$ as before. 
Pick vectors $m_i\in \mathscr M_i$ for $1\leq i\leq N-2$, and states $A,B\in \mathscr V$ in this abstract copy of $\mathscr V$. On taking coinvariants, we get the class
\begin{align} 
\bigl[ \atp{z_1}{m_1} \ox \dots \ox \atp{z_{N-2}}{m_{N-2}} \ox \atp{z_{N-1}}{B} \ox \atp{z_N}{A} \bigr] \in \coinv(\g;\mc A_N; M_1,\dots,M_{N-2},\CC,\CC) 
\end{align}
\begin{rem}\label{rem: interpretation}
One should keep in mind that, despite our attempt to make the notation as similar as possible to the usual case in \cref{coinvariant}, this object encodes a lot of information. 
It is a polynomial differential form on a simplex of dimension $N!-1$, whose pullbacks to certain faces of that simplex encode the behaviour in regimes in which some of the insertion points are ordered in particular ways in the topological direction whenever they collide in the complex plane. 
For example, the pullback to the vertex given by $u_{\sigma} =1$ encodes the behaviour in the regime in which \emph{all} the points are so ordered, in the particular total order $\sigma(1) <\dots<\sigma(N)$. 
\end{rem}

The raviolo vacuum module state-field map captures the behaviour of this coinvariant in the limit in which the $N$th marked point becomes close to the $(N-1)$st marked point, while remaining marked points, the vectors $m_1,\dots,m_{N-2}$, and the states $A, B$ are all held fixed.

Indeed, recall that $\A_N \simeq R\Gamma(\RavConf_N,\O)$ models the dg commutative algebra of derived sections of the structure sheaf on ravioli configuration space. We have the expansion map 
\[ \iota_{z_{N} \to z_{N-1}}: \A_N \to \A_{N-1}\sql z_N-z_{N-1}\sqr\] 
defined in the same way as $\iota_{z_{N+1} \to z_N}$ in \cref{sec: expansion-maps} above.

\begin{thm}[Relation of the raviolo state-field map $Y_\Rav$ to coinvariants]\label{thm: raviolo Y map}
  For all homogeneous states $A,B\in \mathscr V$ and vectors $m_i\in \mathscr M_i$, $1\leq i\leq N-2$, we have 
  \begin{align} &\iota_{z_{N}\to z_{N-1}}
  \bigl[ \atp{z_1}{m_1} \ox \dots \ox \atp{z_{N-2}}{m_{N-2}} \ox \atp{z_{N-1}}{B} \ox \atp{z_N}{A} \bigr] \nn\\
  &\quad\qquad
  = 
  (-1)^{|A||B|}\bigl[ \atp{z_1}{m_1} \ox \dots \ox \atp{z_{N-2}}{m_{N-2}} \ox  Y_\Rav(A;z_N-z_{N-1}) \!\!\atp{z_{N-1}}{B} \bigr] 
  \nn\end{align}
 where \[ Y_\Rav(-;x) : \mathscr V \to \Hom_{\mathbf{dgVect}_\CC}(\mathscr V, \mathscr V\sql x\sqr)\] is the raviolo state-field map defined in \cref{sec: raviolo state-field map}.
  \end{thm}
  
\begin{proof}
  The proof is given in \cref{sec: proof of thm raviolo Y map} below.
\end{proof}

%In \cref{thm: raviolo Y map}, as in \cref{prop: Y map}, it is manifest that the left-hand side of the equality in the statement of the theorem lives in $\A_{N-1}\sql z_N-z_{N-1} \sqr \ox M$, but the fact that the right-hand side does too is something that will emerge in the proof. 

\subsection{Worked example}
Before we give the proof of \cref{thm: raviolo Y map} it is instructive to work through a simple example in detail.
Let $a\in \g$ and consider the state  
\[ \left(a \ox \frac{\dd v}{z}\right) \vac \in \mathscr V\]
in the vacuum module $\mathscr V$. Everything below will hold, mutatis mutandis, with $\dd v$ replaced by $v(1-v)$; cf. the discussion about cochains versus cohomology in \cref{sec: raviolo state-field map} above.

We are first to identify this abstract copy of $\mathscr V$ with the local copy associated to the $N$th marked point, by identifying the coordinate $z$ with the local coordinate $w-z_N$ at that point. Then we are to insert the state above into a generic coinvariant with some vectors $m_1,\dots,m_{N-2}$ at (what are about to be) the far marked points, and some state $B\in \mathscr V$ at (what is about to be) the nearby $N-1$st marked point. We obtain a coinvariant we shall call $f$, 
\begin{align} 
  f:= \bigl[ \atp{z_1}{m_1} \ox \dots \ox \atp{z_{N-2}}{m_{N-2}} \ox \atp{z_{N-1}}{B} \ox \left(a \ox \frac{\dd v}{w-z_N}\right) \atp{z_N}{\vac} \bigr] \in \A_N \ox M. \label{example raviolo coinvariant}
\end{align}
Our aim is to now to ``swap'' the lowering operator onto the other sites. The first step is to identify the lowering operator $a \ox \frac{\dd v}{w-z_N}$ as the expansion at this $N$th site of some element of the global dg Lie algebra. To do that, we  apply the map $g_N$ of \cref{def: gk} to the element 
\[ a \ox \frac{\dd v}{w-z_N} \in \g \ox \CC\sql w-z_N\sqr_- \into \g \ox \A_N\sql w-z_N\sqr_-.\]
We have, cf. \cref{qdef},
\[ q_N^*(v) = \iota_{\{N,N+1\} \subset [1,N+1]} (v) = \sum_{\sigma \in S_{N+1}^{N,N+1}} u_\sigma \]
and so we find that
\[ G:= g_N(a \ox \frac{\dd v}{w-z_N})  = a \ox \frac{\dd \sum_{\sigma \in S_{N+1}^{N,N+1}} u_\sigma}{w-z_N} \in \g\ox\A_{N+1}'\]
where we have introduced a name, $G$, for this element. 
By definition of coinvariants, we have
\begin{align*} - f &= \sum_{i=1}^{N-2} (-)^* \bigl[ \atp{z_1}{m_1} \ox \dots \ox \left(\iota_{w\to z_i} G \right)\atp{z_i}{m_i} \ox \dots \ox \atp{z_{N-2}}{m_{N-2}} \ox \atp{z_{N-1}}{B} \ox \atp{z_N}{\vac} \bigr] \\
&\qquad+ (-)^{|B|} \bigl[  \atp{z_1}{m_1} \ox \dots \ox \atp{z_{N-2}}{m_{N-2}} \ox \left(\iota_{w\to z_{N-1}} G \right)  \atp{z_{N-1}}{B} \ox \atp{z_N}{\vac}    \bigr] \end{align*}
Here $(-)^*$ are certain Koszul signs, which we do not write out in full.

We are to take the expansion $\iota_{z_N\to z_{N-1}}f \in \A_{N-1}\sql z_N-z_{N-1}\sqr \ox M$ of this element $f \in \A_N \ox M$. 
As we shall discuss in \cref{sec: functoriality of raviolo coinvariants} below, coinvariants are suitably functorial, so that
\begin{align} - \iota_{z_{N} \to z_{N-1}} f &= \sum_{i=1}^{N-2}(-)^{*}\bigl[ \atp{z_1}{m_1} \ox \dots \ox \left(\iota_{z_N\to z_{N-1}}\iota_{w\to z_i} g \right)\atp{z_i}{m_i} \ox \dots \ox \atp{z_{N-2}}{m_{N-2}} \ox \atp{z_{N-1}}{B} \ox \atp{z_N}{\vac} \bigr] \nn\\
  &\qquad+ (-1)^{|B|} \bigl[  \atp{z_1}{m_1} \ox \dots \ox \atp{z_{N-2}}{m_{N-2}} \ox \left(\iota_{z_N\to z_{N-1}}\iota_{w\to z_{N-1}} g \right)  \atp{z_{N-1}}{B} \ox \atp{z_N}{\vac}    \bigr] \label{example coinvariant}\end{align}
Now let us actually compute these expansions of $G$ at the other sites. Recall the definition of the expansion map, \cref{defprop: iota}. For every $s\in \{1,\dots,N-1\}$ we have
\begin{align} p_{N+1\to s}^* \left(\sum_{\sigma \in S_{N+1}^{N,N+1}} u_\sigma \right) &= v \sum_{\sigma \in S_N^{N,s}} u_\sigma + (1-v) \sum_{\sigma \in S_N^{N,s}} u_\sigma = \sum_{\sigma \in S_N^{N,s}} u_\sigma \label{example of p}\end{align}
Thus, in particular, the expansion of $G$ at the $(N-1)$st site is given by
\[ \iota_{w\to z_{N-1}} G = - \sum_{k=0}^\8 \frac{ \dd \sum_{\sigma \in S_N^{N,N-1}} u_\sigma}{ (z_N-z_{N-1})^{k+1}} \left(a  \ox (w-z_{N-1})^k\right). \]
%By smoothness of the vacuum module $\mathscr V$, at most finitely many terms in the sum on $k$ here contribute in $(\iota_{w\to z_{N-1}} g) B$, for any given $B\in \mathscr V$.    
We are working in dg $\A_N$-modules, so that the factor $\frac{ \dd \sum_{\sigma \in S_N^{N,N-1}} u_\sigma}{ (z_N-z_{N-1})^{k+1}}$ is a scalar. Let us apply the expansion map $\iota_{z_N\to z_{N-1}}$ to this scalar prefactor. 
Note that  
\[ p_{N\to N-1}^*  \sum_{\sigma \in S_N^{N,N-1}} u_\sigma = (1-u) \sum_{\sigma \in S_{N-1}} u_\sigma  = (1-u).\] 
Here we write $u$ here rather than $v$ for the coordinate of 
\begin{align} \CC\sql z_{N}-z_{N-1} \sqr &:= \bigl\{ \omega \in \CC((z_N-z_{N-1})) \ox \CC[u,\dd u] \nn\\ &\qquad : \omega|_{u=0} \in \CC[[z_N-z_{N-1}]] \text{ and } \omega|_{u=1} \in \CC[[z_N-z_{N-1}]] \bigr\} \nn\end{align}
to avoid a clash with the coordinate $v$ of e.g. $\CC\sql w-z_s \sqr$. Note that the above is consistent with the fact that $\sum_{\sigma \in S_N^{N,N-1}} u_\sigma= \iota_{\{N-1,N\} \subset [1,N]}(1-u)$; see \cref{def: iota j} for the definition of $\iota_{\{N-1,N\} \subset [1,N]}$.
We obtain that 
\[ \iota_{z_N\to z_{N-1}} \frac{ \dd \sum_{\sigma \in S_N^{N,N-1}} u_\sigma}{ (z_N-z_{N-1})^{k+1}} = 1 \frac{\dd (1-u)}{(z_N-z_{N-1})^{k+1}}  \in \A_{N-1}\sql z_N-z_{N-1}\sqr .\]

Next let us consider the expansions of $g$ at the far sites. Crucially, for each $i\leq N-2$, the operations $\iota_{z_N-z_{N-1}}$ and $\iota_{w\to z_i}$ commute. This is true for the the coefficients in $\CC[z_i,w,(z_i-z_j)^{-1},(w-z_i)^{-1}]_{1\leq i,j\leq N; i\neq j} \equiv \B_{N+1}$ just as in the usual case; and for the forms on the simplex, one checks (similarly to \cref{example of p}) that on the one hand
\[ p_{N\to N-1}^* p_{N+1\to i}^* \left(\sum_{\sigma \in S_{N+1}^{N,N+1}} u_\sigma \right) = \sum_{\sigma \in S_N^{N,i}} u_\sigma = \sum_{\sigma\in S_{N-1}^{N-1,i}} u_\sigma\]
while on the other hand
\[ p_{N+1\to i}^* p_{N\to N-1}^*  \left(\sum_{\sigma \in S_{N+1}^{N,N+1}} u_\sigma \right) = \sum_{\sigma \in S_{\{1,\dots,N-1,N+1\}}^{N-1,N+1}} u_\sigma 
= \sum_{\sigma \in S_{N-1}^{N-1,i}} u_\sigma . \]
(We don't get such agreement when $i=N-1$, which is reassuringly consistent with the fact that the usual expansion maps $\iota_{w\to z_{N-1}}$ and $\iota_{z_N\to z_{N-1}}$ certainly do not commute -- indeed this failure is in some sense at the heart of how vertex algebras work.)

Thus, for each $i\leq N-2$, we have
\( \iota_{z_N\to z_{N-1}}\iota_{w\to z_i} G = \iota_{w\to z_i}\iota_{z_N\to z_{N-1}} G 
 \)
and here
\[ \iota_{z_N-z_{N-1}} G 
= + \sum_{k=0}^\8 (z_N-z_{N-1})^k \ox \left( a \otimes \frac{ \dd \sum_{\sigma \in S_{\{1,\dots,N-1,N+1\}}^{N-1,N+1}} u_\sigma}{ (w-z_{N-1})^{k+1}}\right) . \]
%where we used
%\begin{align*} p_{N\to N-1}^* \sum_{\sigma \in S_{N+1}^{N,N+1}} u_\sigma 
%  &= v \sum_{\sigma \in S_{\{1,\dots,N-1,N+1\}}^{N-1,N+1}} u_\sigma + (1-v) \sum_{\sigma \in S_{\{1,\dots,N-1,N+1\}}^{N-1,N+1}} u_\sigma \nn\\
%  &= \sum_{\sigma \in S_{\{1,\dots,N-1,N+1\}}^{N-1,N+1}} u_\sigma \end{align*}
We recognize the terms in this expansion as elements of the global Lie algebra, and by definition of coinvariants we obtain that
\begin{align}  \iota_{z_{N} \to z_{N-1}} f &= 
   \sum_{k=0}^\8  \bigl[ \atp{z_1}{m_1} \ox \dots \ox \atp{z_{N-2}}{m_{N-2}} \ox (-)^{|B|} (z_N-z_{N-1})^k \left( a\ox \frac{\dd v}{ (w-z_{N-1})^{k+1}}\right)\atp{z_{N-1}}{B} \bigr] \nn\\
  &\qquad+ \sum_{k=0}^\8 \bigl[  \atp{z_1}{m_1} \ox \dots \ox \atp{z_{N-2}}{m_{N-2}} \ox (-)^{|B|} \frac{\dd (1-u)}{(z_N-z_{N-1})^{k+1}} \left( a \ox (w-z_{N-1})^k\right)  \atp{z_{N-1}}{B}     \bigr] \nn\end{align}
(Here we dropped the vacuum state at the $N$th site, cf. \cref{sec: propagation of vacua raviolo} below.)

We recognize the first and second lines here as, respectively, the $(-)_+$ and $(-)_-$ parts of the raviolo mode expansion of the state $(a \ox \frac{\dd v}{w-z_N})\vac$ with which we began, as we defined it in \cref{sec: raviolo state-field map}.

\section{Proof of \texorpdfstring{\cref{thm: raviolo Y map}}{Theorem 18}}\label{sec: proof of thm raviolo Y map}
\addtocontents{toc}{\protect\setcounter{tocdepth}{1}}

In this section we prove \cref{thm: raviolo Y map}, namely that the raviolo state-field map $Y_\Rav$ from \cref{sec: raviolo state-field map} emerges naturally when one considers appropriate limits of the spaces of ravioli coinvariants introduced in \cref{sec: induced modules and coinvariants}. 

To separate concerns, we shall first warm up by rehearsing a proof of the analogous statement in the usual case, \cref{prop: Y map}. Our approach is similar to that of \cite{VicedYoungVertexLieAlgebras2017}.

\subsection{Proof of \texorpdfstring{\cref{prop: Y map}}{Theorem 3}}\label{sec: proof: Y map}
%In this section we recall a proof of the relation between the usual state-field map $Y(-,x)(-)$ and the limiting behaviour of rational coinvariants, as we stated it in \cref{prop: Y map}. 

We first need to recall a functoriality property of coinvariants and the property known as propagation of vacua. 

\subsubsection{Functoriality of coinvariants}\label{sec: functoriality of coinvariants}
Let us consider certain spaces of coinvariants with $N-1$ movable marked points. 
The construction of coinvariants in \cref{sec: usual rational conformal blocks} of course goes through with $N-1$ in place of $N$, yielding the $\B_{N-1}$-module
\[ \coinv(\g;\B_{N-1};M_1,\dots,M_{N-1}).\]
But we may also choose to work over $\B_{N}$, or over $\B_{N-1}((z_N-z_{N-1}))$, i.e. to allow our coefficient functions to depend in some prescribed way on the formal variable $z_N$, even though there are now only modules assigned to the points $z_1,\dots,z_{N-1}$. More precisely, we may consider the following cospans of commutative algebras
\begin{equation}
\begin{tikzcd}
 \B_N[w,(w-z_i)^{-1}]'_{1\leq i\leq N-1} \dar[hook] \rar[hook] & \dar[hook]\B_{N-1}((z_N-z_{N-1}))[w,(w-z_i)^{-1}]'_{1\leq i\leq N-1} \\
\bigoplus_{j=1}^{N-1} \B_N((w-z_j)) \rar[hook] &\bigoplus_{j=1}^{N-1} \B_{N-1}((z_N-z_{N-1}))((w-z_j))  \\
\bigoplus_{j=1}^{N-1} \B_N[[w-z_j]] \uar[hook] \rar[hook] & \uar[hook] \bigoplus_{j=1}^{N-1} \B_{N-1}((z_N-z_{N-1}))[[w-z_j]] 
\end{tikzcd}
\label{triples}\end{equation}
-- in $\B_N$-modules on the left and in $\B_{N-1}((z_N-z_{N-1}))$-modules on the right. 
We obtain corresponding spaces of coinvariants which we denote respectively as 
\begin{equation} \coinv(\g;\B_N;M_1,\dots,M_{N-1} ) \quad\text{and}\quad \coinv(\g;\B_{N-1}((z_N-z_{N-1}));M_1,\dots,M_{N-1} ). \nn\end{equation}
Moreover the algebra map  
\begin{equation} \B_N \to \B_{N-1}((z_N-z_{N-1})) \nn\end{equation}  
given by expanding in small $z_N-z_{N-1}$ for fixed $z_1,\dots,z_{N-1}$ allows us to change base ring, in the sense that it induces the embeddings of commutative algebras in $\B_{N-1}$-modules shown as horizontal arrows in the diagram above. 
Let us use $\iotacob$ for that change-of-base map. 
The diagram above commutes. 
In this way, one has the following functoriality property of coinvariants. 

\begin{lem}[Base change commutes with taking coinvariants]\label{lem: functoriality of coinvariants}
The following diagram of $\B_{N-1}$-modules commutes:
\begin{equation}\begin{tikzcd} \bigotimes_{i=1}^{N-1}\MM_i \rar\dar & \bigotimes_{i=1}^{N-1} \iotacob\MM_i \dar \\
\coinv (\g;\B_N;M_1,\dots,M_{N-1} ) \dar{\cong} & \coinv(\g;\B_{N-1}((z_N-z_{N-1}));M_1,\dots,M_{N-1} ) \dar{\cong} \\
\B_N \ox M \rar& \B_{N-1}((z_N-z_{N-1})) \ox M
  \end{tikzcd}\nn
\end{equation}
\qed
\end{lem} 

Let us stress that in the horizontal maps in \cref{triples} above, we are merely performing a change of base ring. 
By contrast, we reserve the notation $ \iota_{z_N\to z_{N-1}}$ for the algebra map which expands in small $z_N-z_{N-1}$ for fixed $z_1,\dots,z_{N-1}$ \emph{and} $w\equiv z_{N+1}$. Thus, for example, 
\be \iota_{z_N\to z_{N-1}}: \B_{N+1}' \to \B_{N-1}[w,(w-z_j)^{-1}]_{1\leq j\leq N-1}((z_N-z_{N-1})) \nn\ee 
sends the element $1/(w-z_N)$ to its expansion $\sum_{k=0}^\8 \frac{(z_N-z_{N-1})^k}{(w-z_{N-1})^{k+1}}$. 

One should keep in mind that the Laurent-expansion maps $\iota_{w\to z_{N-1}}$ and $\iota_{z_N\to z_{N-1}}$ do not commute. For example, they fail to commute when applied to $1/(w-z_N)$. (Indeed, they map from $\B_{N+1}$ into different rings. 
In some sense, this fact is central to the notion of vertex algebras: see e.g. the discussion in \cite[\S1]{FrenkelBenZvi}.)  
On the other hand, for all $i\leq N-2$, the Laurent expansion maps  $\iota_{w \to z_i}$ and $\iota_{z_N \to z_{N-1}}$ do commute.

\subsubsection{Propagation of vacua}\label{sec: propagation of vacua}
When $M_{N} = \CC$, there is a canonical identification, of $\B_N$-modules, between our initial space of coinvariants with $N$ marked points and one with only $N-1$ marked points:\footnote{The reader will notice that while the choice $M_{N}=\CC$ is crucial here, the choice $M_{N-1}=\CC$ actually plays no role. And indeed, the construction goes through more generally, and yields the structure of $\MM_{N-1}$ as a \dfn{module over the vertex algebra} $\VV$.}
\begin{equation} \coinv(\g;\B_N;M_1,\dots,M_{N-2},\CC,\CC) \cong \B_N \ox M \cong \coinv(\g;\B_N;M_1,\dots,M_{N-2},\CC). \nn\end{equation}
One has, moreover, the following equality
\begin{equation} 
\bigl[ \atp{z_1}{m_1} \ox \dots \ox \atp{z_{N-2}}{m_{N-2}} \ox \atp{z_{N-1}}{B} \ox \atp{z_N}{\vac} \bigr] 
= 
\bigl[ \atp{z_1}{m_1} \ox \dots \ox \atp{z_{N-2}}{m_{N-2}} \ox  \atp{z_{N-1}}{B} \bigr]. 
\label{propagation of vacua}\end{equation}
This property is an example of what is sometimes called \dfn{propagation of vacua}. 

\subsubsection{Completion of the proof of \cref{prop: Y map}} 
The equality in \cref{propagation of vacua} establishes the statement of \cref{prop: Y map} in the special case that $A= \vac$ is the vacuum state.

Next we shall show that for all states $A\in \VV$, the class
\begin{equation} \bigl[\atp{z_1}{m_1} \ox \dots \ox \atp{z_{N-2}}{m_{N-2}} \ox \atp{z_{N-1}} B \ox \atp{z_N} A \bigr] \in \coinv(\g;\B_N;M_1,\dots,M_{N-2},\CC,\CC) \cong \B_N \ox M \nn\end{equation}
has a representative of the form
\begin{equation} \sum_i \bigl[A_i^{(-)} \on (\atp{z_1}{m_1} \ox \dots \ox \atp{z_{N-2}}{m_{N-2}}) \ox A_i^{(+)} \on \atp{z_{N-1}}{B} \ox \atp{z_N} \vac \bigr] \nn\end{equation}
for some finite sum over $i$ and for certain $A_i^{(-)}$ and $A_i^{(+)}$ belonging to $U(\g \ox (w-z_N)^{-1}\CC[(w-z_N)^{-1}])$. Here, when we write $A_i^{(-)} \on (m_1 \ox \dots \ox m_{N-2})$, the action is by definition via the embedding 
\begin{align*}
\iotafar: \g \ox (w-z_N)^{-1}\CC[(w-z_N)^{-1}]
  &\into \g \ox (w-z_N)^{-1}\B_N[(w-z_N)^{-1}]\\
  &\into \bigoplus_{i=1}^{N-2} \g \ox \B_N[[w-z_i]].
\end{align*}
Call this embedding $\iotafar$. Similarly when we write $A_i^{(+)} \on B$, the action is via the embedding into $\g \ox \B_N[[w-z_{N-1}]]$ which we shall call $\iotanear$:
\begin{align*}
\iotanear: \g \ox (w-z_N)^{-1}\CC[(w-z_N)^{-1}] &\into \g \ox (w-z_N)^{-1}\B_N[(w-z_N)^{-1}]\\
&\into \g \ox \B_N[[w-z_{N-1}]].
\end{align*}
Indeed, we may suppose
\begin{equation} A = X^1_{-k_1} \dotsm X^n_{-k_n} \vac\nn\end{equation}
for some number $n\in \ZZ_{\geq 0}$ of elements $X^i\in\g$ and mode numbers $-k_i\in \ZZ_{< 0}$. Here $X_k \coloneq X\ox (w-z_N)^k$. (Such states $A$ span $\VV$ as a $\CC$-vector space.)
By a straightforward induction on $n$, one checks that 
\begin{align}& \bigl[\atp{z_1}{m_1} \ox \dots \ox \atp{z_{N-2}}{m_{N-2}} \ox \atp{z_{N-1}} B \ox \atp{z_N} A \bigr] \nn\\
&= (-1)^n \sum_{m=0}^n   \sum_{(\mu,\nu) \in \Unshf m n}  \bigl[ (\iotafar \overleftarrow A_\mu) (\atp{z_1}{m_1} \ox \dots \ox \atp{z_{N-2}}{m_{N-2}}) \ox (\iotanear\overleftarrow A_\nu) \!\!\atp{z_{N-1}}{B} \ox \atp{z_N} \vac \bigr] \label{swapping}
\end{align}
where the inner sum is over unshuffles, i.e. permutations $(\mu_1,\dots,\mu_m,\nu_1,\dots,\nu_{n-m})$ of $(1,\dots,n)$ such that $\mu_1<\dots<\mu_m$ and $\nu_1<\dots<\nu_{n-m}$, and where we write 
\begin{equation} \overleftarrow A_\mu \coloneq X^{\mu_m}_{-k_{\mu_m}} \dotsm X^{\mu_1}_{-k_{\mu_1}} ,\qquad
    \overleftarrow A_\nu \coloneq X^{\nu_{n-m}}_{-k_{\nu_{n-m}}} \dotsm X^{\nu_1}_{-k_{\nu_1}} .\nn\end{equation}
Then by propagation of vacua as in \cref{propagation of vacua}, we may regard the right-hand side in \cref{swapping} as an element %\footnote{Namely, in this case, as the class of $(-1)^n \sum_{m=0}^n \sum_{(\mu,\nu) \in \Unshf m n}  (\iotafar \overleftarrow A_\mu) (m_1\ox\dots\ox m_{N-2}) \ox  (\iotanear\overleftarrow A_\nu) B \in \MM_1\ox\dots\MM_{N-2} \ox \VV$.}
of the space of coinvariants $\coinv(\g;\B_N;M_1,\dots,M_{N-2},\CC)$. 
The equality in \cref{swapping} is in $\B_N \ox M$. We may apply the change-of-base map $\iotacob$ to both sides to obtain the equality
\begin{align}& \iotacob \bigl[\atp{z_1}{m_1} \ox \dots \ox \atp{z_{N-2}}{m_{N-2}} \ox \atp{z_{N-1}} B \ox \atp{z_N} A \bigr] \label{iotadswapping}\\
&=  (-1)^n \sum_{m=0}^n   \sum_{(\mu,\nu) \in \Unshf m n}  \iotacob\bigl[(\iotafar\overleftarrow A_\mu) (\atp{z_1}{m_1} \ox \dots \ox \atp{z_{N-2}}{m_{N-2}}) \ox ( \iotanear\overleftarrow A_\nu) \atp{\!\!\!z_{N-1}\!\!\!}{B}  \bigr]\nn\\
&=  (-1)^n \sum_{m=0}^n   \sum_{(\mu,\nu) \in \Unshf m n}  \bigl[\iotacob(\iotafar\overleftarrow A_\mu) (\atp{z_1}{m_1} \ox \dots \ox \atp{z_{N-2}}{m_{N-2}}) \ox \iotacob(\iotanear\overleftarrow A_\nu) \atp{\!\!\!z_{N-1}\!\!\!}{B}  \bigr]\nn\\
&=  (-1)^n \sum_{m=0}^n   \sum_{(\mu,\nu) \in \Unshf m n}  \bigl[(\iotacob\iotafar\overleftarrow A_\mu) \iotacob(\atp{z_1}{m_1} \ox \dots \ox \atp{z_{N-2}}{m_{N-2}}) \ox \iotacob(\iotanear\overleftarrow A_\nu) \atp{\!\!\!z_{N-1}\!\!\!}{B}  \bigr]\nn
\end{align}
in $\B_{N-1}((z_N-z_{N-1}))\ox M$. In the second step, we used the functoriality of coinvariants, \cref{lem: functoriality of coinvariants}.

It remains to show that this expression is equal to the expression on the right-hand side in the statement of  \cref{prop: Y map}. The latter is, first and foremost, a formal series in $(z_N-z_{N-1})^{\pm 1}$ whose coefficients belong to the space of coinvariants $\coinv(\g;\B_{N-1};M_1,\dots,M_{N-2},\CC) \cong_{\B_{N-1}} \B_{N-1} \ox M$:
\begin{align} 
&\bigl[ \atp{z_1}{m_1} \ox \dots \ox \atp{z_{N-2}}{m_{N-2}} \ox  Y(A;z_N-z_{N-1}) \!\!\atp{z_{N-1}}{B} \bigr] \nn\\
&
\qquad= \sum_{k\in \ZZ} (z_N-z_{N-1})^{-k-1}  \bigl[ \atp{z_1}{m_1} \ox \dots \ox \atp{z_{N-2}}{m_{N-2}} \ox A_{(k)} \!\!\atp{z_{N-1}}{B} \bigr] 
\nn\end{align}
Smoothness of the module $\VV$ ensures that for each fixed $A,B\in \VV$ this is in fact a formal Laurent series, i.e. $A_{(k)} B = 0$ for $k\gg 0$. Thus, it certainly belongs in $(\B_{N-1} \ox M)((z_N-z_{N-1}))$. 
To show that it is equal to the expression in \cref{iotadswapping} we must show that for each $k\in \ZZ$, the coefficients of $(z_N-z_{N-1})^k$ agree. Consider any term $(\mu,\nu)$ in the sum in \cref{iotadswapping}. 
We have 
\[ \iotacob \iotafar \overleftarrow A_\mu = \iota_{z_N\to z_{N-1}} \iotafar \overleftarrow A_\mu =  \iotafar \iota_{z_N\to z_{N-1}} \overleftarrow A_\mu. \]

Since $(\iotanear\overleftarrow A_\nu) B \in (z_N-z_{N-1})^{-1}\B_{N-1}[(z_N-z_{N-1})^{-1}]\ox \VV$, only finitely many terms in the series $\iota_{z_N\to z_{N-1}} \overleftarrow A_\mu$ contribute to the overall coefficient of $(z_N-z_{N-1})^k$.  The coefficients of these finitely many terms belong to $U(\g \ox \B_{N-1}[w,(w-z_{N-1})^{-1}]')$, and we can swap them over to the module at the marked point $z_{N-1}$, by definition of the space of coinvariants. After doing so we obtain
\begin{align}&  (-1)^n \sum_{m=0}^n (-1)^m  \!\!\!\!\sum_{\!\!(\mu,\nu) \in \Unshf m n\!\!}  \bigl[\atp{z_1}{m_1} \ox \dots \ox \atp{z_{N-2}}{m_{N-2}} \ox 
(\iotanear \iota_{z_N\to z_{N-1}} \overrightarrow A_\mu) (\iotacob \iotanear\overleftarrow A_\nu) \atp{\!\!\!z_{N-1}\!\!\!}{B} \bigr]\nn
\end{align}
where $\overrightarrow A_\mu \coloneq X^{\mu_1}_{-k_{\mu_1}} \dots X^{\mu_m}_{-k_{\mu_m}}$. Here, we recognize the state-field map $Y$:
\begin{equation} \sum_{m=0}^n (-1)^{n-m}  \sum_{(\mu,\nu) \in \Unshf m n} (\iotanear \iota_{z_N\to z_{N-1}} \overrightarrow A_\mu) (\iotacob \iotanear\overleftarrow A_\nu)B = Y(A,z_N-z_{N-1})B .\nn\end{equation}
(This expression can be checked by induction on $n$.)

\subsection{Proof of \texorpdfstring{\cref{thm: raviolo Y map}}{Theorem 16}} 
Now we return to the proof of \cref{thm: raviolo Y map}. 
Where possible, we shall try to follow the proof above almost word-for-word.

We begin by establishing a raviolo analogue of the functoriality property of coinvariants, \cref{lem: functoriality of coinvariants}. 

\subsubsection{Functoriality of raviolo coinvariants}\label{sec: functoriality of raviolo coinvariants}
Once again, let us consider certain spaces of coinvariants with $N-1$ movable marked points. 
We certainly have the dg $\A_{N-1}$-module
\[ \coinv(\g;\A_{N-1};M_1,\dots,M_{N-1}).\]
But we may also choose to work over $\A_{N}$, or over $\A_{N-1}\sql z_N-z_{N-1}\sqr$. More precisely, we may consider the following cospans of commutative algebras
\begin{equation}
\begin{tikzcd}
 \A_N\{w,(w-z_i)^{-1}\}'_{1\leq i\leq N-1} \dar[hook] \rar[hook] & \dar[hook]\A_{N-1}\sql z_N-z_{N-1}\sqr\{w,(w-z_i)^{-1}\}'_{1\leq i\leq N-1} \\
\bigoplus_{j=1}^{N-1} \A_N\sql w-z_j\sqr \rar[hook] &\bigoplus_{j=1}^{N-1} \A_{N-1}\sql z_N-z_{N-1}\sqr\sql w-z_j\sqr  \\
\bigoplus_{j=1}^{N-1} \A_N[[w-z_j]] \uar[hook] \rar[hook] & \uar[hook] \bigoplus_{j=1}^{N-1} \A_{N-1}\sql z_N-z_{N-1}\sqr[[w-z_j]] 
\end{tikzcd}
\label{raviolo triples}\end{equation}
-- in dg $\A_N$-modules on the left and in dg $\A_{N-1}\sql z_N-z_{N-1}\sqr$-modules on the right. Here, the algebras appearing in the top line are defined in close analogy with the definition of $\A_{N+1}'$ in \cref{sec: expansion-maps}. Namely, we first let $\A_N\{w,(w-z_i)^{-1}\}_{1\leq i\leq N-1}$ denote the commutative algebra in dg $\A_N$-modules given by the subalgebra of $\A_{N+1}$ consisting of forms which are regular in $w-z_N\equiv z_{N+1}-z_N$ everywhere:
\begin{align} &\A_N\{w,(w-z_i)^{-1}\}_{1\leq i\leq N-1}\nn\\ &\coloneq \bigl\{ \omega \in \B_N[w,(w-z_i)^{-1}]_{1\leq i\leq N-1} \ox \CC[u_\sigma, \dd u_\sigma]_{\sigma \in S_{N+1}}\big/ \langle \sum_{\sigma\in S_{N+1}} u_\sigma - 1, \sum_{\sigma\in S_{N+1}} \dd u_\sigma \rangle \nn\\
  &\qquad:\text{for all distinct $i,j\in [1,N+1]$, the pullback $\omega|_{\{u_\sigma=0 \forall \sigma\in S_{N+1}^{ij}\}}$ }\nn\\&\qquad\qquad\text{is regular in $z_i-z_j$ } \bigr\}.
\nn\end{align}
Similarly, we let $\A_{N-1}\sql z_N-z_{N-1}\sqr\{w,(w-z_i)^{-1}\}_{1\leq i\leq N-1}$ denote the dg commutative algebra given by 
\begin{align} &\A_{N-1}\sql z_N-z_{N-1}\sqr\{w,(w-z_i)^{-1}\}_{1\leq i\leq N-1}\nn\\ 
  &\coloneq \bigl\{ \omega \in \B_{N-1}((z_N-z_{N-1}))[w,(w-z_i)^{-1}]_{1\leq i\leq N-1}\ox \CC[v,\dd v] \nn\\&\qquad\qquad\qquad 
     \ox \CC[u_\sigma, \dd u_\sigma]_{\sigma \in S_{[1,N-1]\cup\{N+1\}}}\big/ \langle \sum_{\!\!\!\sigma\in S_{[1,N-1]\cup\{N+1\}}\!\!\!} u_\sigma - 1, \sum_{\!\!\!\sigma\in S_{[1,N-1]\cup\{N+1\}}\!\!\!} \dd u_\sigma \rangle \nn\\
  &\qquad:\text{for all distinct $i,j\in [1,N-1]\cup\{N+1\}$,}\nn\\&\qquad\qquad\text{the pullback $\omega|_{\{u_\sigma=0 \forall \sigma\in S_{[1,N-1]\cup\{N+1\}}^{ij}\}}$ is regular in $z_i-z_j$ ,} \nn\\
  &\qquad\quad\text{and the pullbacks $\omega|_{v=0}$ and $\omega|_{v=1}$ are both regular in $z_N-z_{N-1}$} \bigr\}.
\nn\end{align}
It is a commutative algebra in dg $\A_{N-1}\sql z_N-z_{N-1}\sqr$-modules, by an argument similar to that in \cref{lem: AN action on AN+1}.

We again let prime $'$ denote the subalgebras of these algebras consisting of forms $\omega$ such that $\omega\to 0$ as $w\to \8$. 

We obtain corresponding spaces of coinvariants which we denote respectively as
\[ \coinv(\g;\A_N;M_1,\dots,M_{N-1} ) \quad\text{and}\quad \coinv(\g;\A_{N-1}\sql z_N-z_{N-1}\sqr;M_1,\dots,M_{N-1} ). \]
Moreover the algebra map  
\[ \A_N \to \A_{N-1}\sql z_N-z_{N-1}\sqr \nn\]
from \cref{sec: expansion-maps} induces the maps of commutative algebras in dg $\A_{N-1}$-modules shown as horizontal arrows in the diagram above. At the cost of overloading notation somewhat, let us continue to use $\iotacob$ for that change-of-base map.

The diagram above is then a commuting diagram in the category of commutative algebras in dg $\A_{N-1}$-modules. 
In this way, one has the following functoriality property of coinvariants. 

\begin{lem}[Base change commutes with taking coinvariants -- raviolo case]\label{lem: functoriality of raviolo coinvariants}
The following diagram of dg $\A_{N-1}$-modules commutes:
\begin{equation}
\begin{tikzcd} \bigotimes_{i=1}^{N-1}\mathscr M_i \rar\dar & \bigotimes_{i=1}^{N-1} \iotacob \mathscr M_i \dar \\
\coinv (\g;\A_N;M_1,\dots,M_{N-1} )\rar & \coinv(\g;\A_{N-1}\sql z_N-z_{N-1}\sqr;M_1,\dots,M_{N-1} )  
  \end{tikzcd}\nn
\end{equation}
\qed
\end{lem}

We continue to reserve the notation $ \iota_{z_N\to z_{N-1}}$ for the expansion map, in the sense of \cref{sec: expansion-maps}, which expands in small  $z_N-z_{N-1}$ for fixed $z_1,\dots,z_{N-1}$ \emph{and} $w\equiv z_{N+1}$. 
%Thus, for example, we have the map
%\[ \iota_{z_N\to z_{N-1}}: \A_{N+1}' \to \A_{N-1}\{w,(w-z_j)^{-1}\}'_{1\leq j\leq N-1}\sql z_N-z_{N-1}\sqr. \]
%The definition of $\A_{N-1}\{w,(w-z_j)^{-1}\}'_{1\leq j\leq N-1}$ is analogous to that of $\A_{N+1}'$ in \cref{sec: expansion-maps}.

\subsection{Propagation of vacua in the raviolo case}\label{sec: propagation of vacua raviolo}
When $M_N = \CC$ is the trivial module, there is an isomorphism, of dg $\A_N$-modules, between our initial space of coinvariants with $N$ marked points and one with only $N-1$ marked points:
\begin{equation} \coinv(\g;\A_N;M_1,\dots,M_{N-2},\CC,\CC) \cong \coinv(\g;\A_N;M_1,\dots,M_{N-2},\CC). \nn\end{equation}
One has the equality
\begin{equation}
\bigl[ \atp{z_1}{m_1} \ox \dots \ox \atp{z_{N-2}}{m_{N-2}} \ox \atp{z_{N-1}}{B} \ox \atp{z_N}{\vac} \bigr]
=
\bigl[ \atp{z_1}{m_1} \ox \dots \ox \atp{z_{N-2}}{m_{N-2}} \ox  \atp{z_{N-1}}{B} \bigr],
\label{propagation of vacua raviolo}\end{equation}
which is the \dfn{propagation of vacua} property in the raviolo case.

\subsection{Completion of the proof of \texorpdfstring{\cref{thm: raviolo Y map}}{Theorem 16}}\label{sec: proof: raviolo Y map}
The equality in \cref{propagation of vacua raviolo} establishes the statement of \cref{thm: raviolo Y map} in the special case that $A= \vac$ is the vacuum state.

Next we shall show that for all states $A\in \mathscr V$ in the raviolo vacuum module $\mathscr V$, the class
\begin{equation} \bigl[\atp{z_1}{m_1} \ox \dots \ox \atp{z_{N-2}}{m_{N-2}} \ox \atp{z_{N-1}} B \ox \atp{z_N} A \bigr] \in  \A_N \ox M \nn\end{equation}
has a representative of the form
\begin{equation} \sum_i \bigl[A_i^{(-)} \on (\atp{z_1}{m_1} \ox \dots \ox \atp{z_{N-2}}{m_{N-2}}) \ox A_i^{(+)} \on \atp{z_{N-1}}{B} \ox \atp{z_N} \vac \bigr] \nn\end{equation}
for some finite sum over $i$ and for certain $A_i^{(-)}$ and $A_i^{(+)}$ belonging to $U(\g\ox \CC\sql w-z_N\sqr_-)$. 
%(Recall the definition of $\CC\sql z\sqr_-$ from \cref{sec: vacuum-module}.) 
Recall the algebra map \[g_N: \A_N\sql w-z_N\sqr_- \to \A_{N+1}'\] from \cref{def: gk}.
When we write $A_i^{(-)} \on (m_1 \ox \dots \ox m_{N-2})$, the action is by definition via the map of dg Lie algebras
\begin{align} \iotafar: \g\ox \CC\sql w-z_N\sqr_- &\into \g \ox \A_N\sql w-z_N\sqr_- \nn\\
    &\xrightarrow{g_N} \g \ox \A_{N+1}' \xrightarrow{(\iota_{w\to z_1},\dots,\iota_{w\to z_{N-2}})} \bigoplus_{i=1}^{N-2} \g \ox \A_N\sql w-z_i\sqr_+,\nn\end{align}
which we continue to call $\iotafar$. Similarly when we write $A_i^{(+)} \on B$, the action is by definition via the map of dg Lie algebras
\begin{align} \iotanear: \g\ox \CC\sql w-z_N\sqr_- &\into \g \ox \A_N\sql w-z_N\sqr_- \nn\\
  &\xrightarrow{g_N} \g \ox \A_{N+1}' \xrightarrow{\iota_{w\to z_{N-1}}} \g \ox \A_N\sql w-z_{N-1}\sqr_+.\nn\end{align}
Indeed, we may suppose
\[ A = X^1 \dotsm X^n \vac\nn\]
for some number $n\geq 0$ of elements $X^i \in \g \ox \CC\sql w-z_N\sqr_-$, $1\leq i\leq n$. By a straightforward induction on $n$, one checks that
\begin{align}& \bigl[\atp{z_1}{m_1} \ox \dots \ox \atp{z_{N-2}}{m_{N-2}} \ox \atp{z_{N-1}} B \ox \atp{z_N} A \bigr] \nn\\
&= \sum_{m=0}^n   \sum_{(\mu,\nu) \in \Unshf m n} (-1)^{n+\chi}  \bigl[ (\iotafar \overleftarrow A_\mu) (\atp{z_1}{m_1} \ox \dots \ox \atp{z_{N-2}}{m_{N-2}}) \ox (\iotanear\overleftarrow A_\nu) \!\!\atp{z_{N-1}}{B} \ox \atp{z_N} Y\vac \bigr] \label{swapping raviolo}
\end{align}
where the inner sum is over unshuffles, as defined above after \eqref{swapping}, and where we write
\begin{equation} \overleftarrow A_\mu \coloneq X^{\mu_m} \dotsm X^{\mu_1} ,\qquad
    \overleftarrow A_\nu \coloneq X^{\nu_{n-m}} \dotsm X^{\nu_1} .\nn\end{equation}
 In the expression above $(-1)^\chi$ denotes the appropriate Koszul sign coming from the braiding of the tensor product; it is implicitly a function of the grades of the factors $X^i$ and of the states $m_j$ (all of which without loss of generality we shall assume are homogeneous) and on the unshuffle $(\mu,\nu)$. We don't need to work it out explicitly at this stage -- many of the terms will cancel out in the next swapping step below; in particular the dependence on the $|m_j|$ will drop out.

By propagation of vacua as in \cref{propagation of vacua raviolo}, we may regard the right-hand side in \cref{swapping raviolo} as an element of the space of coinvariants $\coinv(\g;\A_N;M_1,\dots,M_{N-2},\CC,\CC)$. We may apply the change-of-base map $\iotacob$ to both sides to obtain the equality
\begin{align}& \iotacob  \bigl[\atp{z_1}{m_1} \ox \dots \ox \atp{z_{N-2}}{m_{N-2}} \ox \atp{z_{N-1}} B \ox \atp{z_N} A \bigr] \label{iotadswapping raviolo}\\
&=  \sum_{m=0}^n   \sum_{(\mu,\nu) \in \Unshf m n} (-1)^{n+\chi} \iotacob\bigl[(\iotafar\overleftarrow A_\mu) (\atp{z_1}{m_1} \ox \dots \ox \atp{z_{N-2}}{m_{N-2}}) \ox ( \iotanear\overleftarrow A_\nu) \atp{\!\!\!z_{N-1}\!\!\!}{B}  \bigr]\nn\\
&=  \sum_{m=0}^n   \sum_{(\mu,\nu) \in \Unshf m n} (-1)^{n+\chi} \bigl[\iotacob(\iotafar\overleftarrow A_\mu) (\atp{z_1}{m_1} \ox \dots \ox \atp{z_{N-2}}{m_{N-2}}) \ox \iotacob(\iotanear\overleftarrow A_\nu) \atp{\!\!\!z_{N-1}\!\!\!}{B}  \bigr]\nn\\
&=  \sum_{m=0}^n   \sum_{(\mu,\nu) \in \Unshf m n} \hskip-16pt (-1)^{n+\chi}  \bigl[(\iotacob\iotafar\overleftarrow A_\mu) \iotacob(\atp{z_1}{m_1} \ox \dots \ox \atp{z_{N-2}}{m_{N-2}}) \ox \iotacob(\iotanear\overleftarrow A_\nu) \atp{\!\!\!z_{N-1}\!\!\!}{B}  \bigr]\nn.
\end{align}
In the second step, we used \cref{lem: functoriality of raviolo coinvariants}. 

The space of coinvariants $\coinv(\g, \A_{N-1}\sql z_N-z_{N-1}, M_1,\dots,M_{N-2},\CC,\CC)$ is a quotient of the free module $\A_{N-1}\sql z_N-z_{N-1}\sqr \ox M$, cf. \cref{def: coinv}. Consider a representative in that free module. An element of $\A_{N-1}\sql z_N-z_{N-1}\sqr \ox M$ is by definition an element of 
\[ \A_{N-1}((z_N-z_{N-1})) \ox \CC[v,\dd v] \ox M \]
whose pullbacks to $v=0$ and $v=1$ are regular in $z_N-z_{N-1}$. So it is a Laurent series in $z_N-z_{N-1}$ (albeit one obeying certain extra conditions), and thus to specify it it is enough to give the coefficient of $(z_N-z_{N-1})^k$ for every $k\in \ZZ$.
So, let $k\in \ZZ$ and consider the coefficent of $(z_N-z_{N-1})^k$ in the expression in \cref{iotadswapping raviolo}. Consider any term $(\mu,\nu)$ in the sum. 
We still have\footnote{In more detail: it is still the case that $\iota_{w\to z_s}$ and $\iota_{z_N \to z_{N-1}}$ commute whenever $s\leq N-2$, where these are now the expansion maps defined as in \cref{defprop: iota}. 
Indeed, one checks that the maps $p_{N+1\to s}^*$ and $p_{N\to N-1}^*$ (the latter defined by obvious analogy with the former) commute for all $s\leq N-2$. It is interesting to note also that they do \emph{not} commute for $s=N-1$.}
\[ \iotacob \iotafar \overleftarrow A_\mu = \iota_{z_N\to z_{N-1}} \iotafar \overleftarrow A_\mu =  \iotafar \iota_{z_N\to z_{N-1}} \overleftarrow A_\mu. \]

By smoothness of the vacuum module $\mathscr V$, we have that
\[ \iotanear \overleftarrow A_\nu B \in \A_{N-1}\sql z_N-z_{N-1}\sqr_- \ox \mathscr V . \]
(To stress the point: on grading grounds, at most finitely many terms in the series $\iotanear X\in \g \ox \A_N[[w-z_{N-1}]]$ are nonzero when acting on $B$, for any given state $B\in \mathscr V$ and $X\in \g \ox \CC\sql w-z_N\sqr$. It follows that there is a lower bound on the powers of $z_N-z_{N-1}$ that appear in $\iotanear \overleftarrow A_\nu B$. This logic is exactly as in the usual case discussed in \cref{sec: proof: Y map}.)

Therefore only finitely many terms in the series $\iota_{z_N\to z_{N-1}} \overleftarrow A_\mu$ contribute to the overall coefficient of $(z_N-z_{N-1})^k$.  The coefficients of these finitely many terms belong to $U(\g \ox \A_{N-1}\sql z_N-z_{N-1}\sqr_-)$, and we can swap them over to the module at the marked point $z_{N-1}$, by definition of the space of coinvariants. After doing so we obtain 
\begin{equation}  (-1)^{|A||B|}\bigl[\atp{z_1}{m_1} \ox \dots \ox \atp{z_{N-2}}{m_{N-2}} \ox 
Y_\Rav(A, z_N-z_{N-1})  \atp{\!\!\!z_{N-1}\!\!\!}{B} \bigr]\nn
\end{equation}
where we recognize the expression for the raviolo state-field map $Y_\Rav$ from \cref{lem: state-field recursion}.
This completes the proof of \cref{thm: raviolo Y map}.

\appendix
\section{Semisimplicial objects and the Thom-Sullivan functor}\label{sec: review of thom-sullivan}
\subsection{Semisimplicial objects}\label{sec: simplicial sets}
Let $\Delta$ denote the category whose objects are the finite nonempty totally-ordered sets 
\begin{equation} [n] := \{0<1<\dots<n\}, \qquad n\in \ZZ_{\geq 0}, \nn\end{equation}
and whose morphisms are the strictly %weakly 
order-preserving maps $\theta: [n] \to [N]$. 
Such maps are generated by \dfn{coface maps},
\begin{equation} d_j : [n] \to [n+1];\quad i\mapsto \begin{cases} i & i<j \\ i+1 & i\geq j \end{cases} \quad\text{for}\quad j=0,1,\dots,n+1.\nn\end{equation}
% and \emph{codegeneracy maps},
% \begin{equation} s_j : [n] \to [n-1];\quad i\mapsto \begin{cases} i & i<j \\ i-1 & i\geq j \end{cases} \quad\text{for}\quad
% j=1,\dots,n.\nn\end{equation}  
% (These maps obey certain relations; see e.g. \cite[\S8]{Weibel}.) 
% We write
% \begin{equation} 
% \begin{tikzcd} 
%  \dots \quad \left[2\right]
% \rar[<-,shift left=8pt]\rar[<-,shift left=4pt]\rar[<-]\rar[shift right = 4pt]\rar[shift right = 8pt]& 
% \left[1\right] \rar[<-,shift left=4pt]\rar[<-] \rar[shift right=4pt]& \left[0\right]
% \end{tikzcd}.
% \nn\end{equation}
A \dfn{semicosimplicial object} $A$ in a category $\C$ is a functor $A: \Delta \to \C$.
Similarly,
a \dfn{semisimplicial object} $Z$ in a category $\C$ is a functor $Z: \Delta^\op\to \C$.   
%In particular, a \dfn{semisimplicial set} $S: \Delta^\op \to \Set$ is a semisimplicial object $S$ in the category of sets. %For each $n$, $S([n])$ is called the set of $n$-simplices of $S$. 
The maps $\del^n_i := Z(d^n_i) : Z([n+1]) \to Z([n])$ are the \dfn{face maps} of $Z$. 
One thinks of the category $\Delta$ as follows:
\be 
\begin{tikzcd} 
\dotsb {[2]}
\rar[<-,shift left=-4pt]\rar[<-,shift left=4pt]\rar[<-] & 
{[1]} \rar[<-,shift left=2pt]\rar[<-,shift right=2pt]& {[0]}
\end{tikzcd}
\nn\ee
and so a semisimplicial object $Z$ in $\C$ defines a diagram in $\C$ of the form
\be
\begin{tikzcd}
\dotsb Z([2]) \rar[->,shift left=-4pt]\rar[->,shift left=4pt]\rar[->] &
Z([1]) \rar[->,shift left=2pt]\rar[->,shift right=2pt]& Z([0]).
\end{tikzcd}
\nn\ee

\subsection{Polynomial differential forms on the standard algebro-geometric simplex}\label{sec: omega}
There is a simplicial dg commutative algebra 
\begin{equation} \Omega:\Delta^\op \to \dgCAlg \nn\end{equation} 
defined as follows. For each $n\geq 0$, $\Omega([n])$ is the dg commutative algebra
\begin{equation} \Omega([n]) :=  \CC[t_0,\dots,t_n; \dd t_0, \dots \dd t_n]\big/\langle \sum_{i=0}^n t_i -1, \sum_{i=0}^n \dd t_i \rangle \nn\end{equation} 
with $t_i$ in degree $0$ and $\dd t_i$ in degree $1$, for each $i$, and equipped with the usual de Rham differential.  
For any map $\phi: [n] \to [N]$ of $\Delta$, 
\begin{equation} \Omega(\phi) : \Omega([N]) \to \Omega([n]) \nn\end{equation}
is the map of dg commutative algebras defined by $t_i \mapsto \sum_{j\in \phi^{-1}(i)} t_j$.
One should think of $\Omega([n])$ as the complex of polynomial differential forms on the \dfn{standard algebro-geometric $n$-simplex},
\begin{equation} \Delta_\CC^n := \Spec \CC[t_0,\dots,t_n]/\langle \sum_{i=0}^n t_i -1 \rangle \into \AA_\CC^{n+1}.
%  , 
\nn\end{equation}
%which is a semicosimplicial object in affine schemes.

\subsection{The functor \texorpdfstring{$\Th$}{Th}}\label{sec: Thom-Sullivan functor}
Suppose we are given a functor $A : \Delta \to \CAlg(\Vect_\CC)$; that is, suppose we are given a semicosimplicial object in commutative algebras in vector spaces. One can construct a commutative algebra in dg vector spaces, given by the graded vector space
\begin{align} \Th^\bul(A)  &:= \biggl\{ \mathbf a = (a_n)_{n\geq 0 } \in \prod_{n\geq 0} A([n]) \ox \Omega^\bul([n]) :\nn\\&\qquad\qquad\qquad
\left(A(\phi) \ox \id \ox\right) a_n
= \left(\id\ox\Omega(\phi)\right) a_m \quad \text{in}\quad \Omega([n]) \ox A([m])  \nn\\&\qquad\qquad\qquad
\qquad\text{for all maps $\phi:[n] \to [m]$ of $\Delta$} \biggr\},\label{def: Thom-Sullivan functor}
\end{align}
equipped with the de Rham differential $\dd = \dd_{\text{de Rham}} \ox \id$ and the graded commutative product given by $(\omega \ox a)( \tau \ox b) :=  \omega \wedge \tau \ox ab$.
This defines the action on objects of a functor, called the \dfn{Thom-Sullivan} \cite{HS} or \dfn{Thom-Whitney} \cite{FMM}  functor, from semicosimplicial commutative algebras to dg commutative algebras,
\begin{equation} \Th:[\Delta, \CAlg(\Vect_\CC)] \to \dgCAlg. \nn\end{equation}
There is a quasi-isomorphism of dg vector spaces $\int:\Th^\bul(A) \xrightarrow\sim \C^\bul(A)$ to the unnormalized cochain complex $\C^\bul(A)$ associated $A$, namely the complex with $C^n(A) := A([n])$ for $n\geq 0$, $C^n(A) = 0$ for $n<0$, and differential $\dd_\C^n := \sum_{j=0}^{n+1} (-1)^j d_j$. (This quasi-isomorphism is defined by integrating over the simplices; see \cite[\S5.2.6]{HS}.) 

(By suitably totalizing, the definition $\Th^\bul$ extends to a functor from semicosimplicial dg commutative algebras to dg commutative algebras, which is how $\Th^\bul$ is more commonly presented; but the semicosimplicial commutative algebras we encounter in the present paper are all concentrated in degree $0$, so \cref{def: Thom-Sullivan functor} suffices for our purposes.)

\printbibliography
%\bibliography{bibliography}
%\bibliographystyle{amsalpha}

\end{document}

% Local Variables:
% TeX-command-extra-options: "-shell-escape"
% End:

%********************************